\documentclass[12pt,a4paper,twoside]{article}
\usepackage[hmarginratio=1:1,bmargin=1.3in]{geometry}
\usepackage[frenchb]{babel}
\usepackage[utf8]{inputenc}
\usepackage[OT2,T1]{fontenc}
\usepackage{amsmath, amsthm}
\usepackage{amsfonts}
\usepackage{amssymb}
\usepackage[all]{xy}
\usepackage{graphicx}
\usepackage[nottoc, notlof, notlot]{tocbibind}
\usepackage{array}
\usepackage{url}
\usepackage{rotating}
\usepackage{fancyhdr}
\usepackage{fancyhdr}
\usepackage{titlesec}
\usepackage{xfrac} 
\titleformat{\section}[hang]{\center\Large\bf}{\thesection.}{0.5cm}{}

\DeclareSymbolFont{cyrletters}{OT2}{wncyr}{m}{n}
\DeclareMathSymbol{\Sha}{\mathalpha}{cyrletters}{"58}
\DeclareMathSymbol{\Brusse}{\mathalpha}{cyrletters}{"42}
\addto\captionsfrenchb{}

\theoremstyle{plain}
\newtheorem{theorem}{Th\'eor\`eme}[section]
\newtheorem{lemma}[theorem]{Lemme}
\newtheorem{proposition}[theorem]{Proposition}
\newtheorem{corollary}[theorem]{Corollaire}
\newtheorem{definition}[theorem]{D\'efinition}
\newtheorem{thmx}{Théorème}

\theoremstyle{definition}
\newtheorem{remarque}[theorem]{Remarque}

\newtheorem{notation}[theorem]{Notation}
\newtheorem{question}[theorem]{Question}

\newtheoremstyle{hypo}  
  {\topsep}   
  {\topsep}   
  {\itshape}  
  {1.5ex}       
  {\bfseries} 
  {)}         
  {8pt plus 1pt minus 1pt}  
  {}          
\theoremstyle{hypo}

\newtheoremstyle{hypol}  
  {\topsep}   
  {\topsep}   
  {\itshape}  
  {1.5ex}       
  {\bfseries} 
  {)$_{\ell}$}         
  {8pt plus 1pt minus 1pt}  
  {}          
\theoremstyle{hypol}

\newtheoremstyle{DP}  
  {\topsep}   
  {\topsep}   
  {\itshape}  
  {}       
  {\bfseries} 
  {.}         
  {8pt plus 1pt minus 1pt}  
  {}          
\theoremstyle{DP}

\begin{document}

\title{\textbf{\scshape Autour d'une conjecture de Kato et Kuzumaki}}

\author{Diego Izquierdo\\
\small
École Normale Supérieure\\
\small
45, Rue d'Ulm - 75005 Paris - France\\
\small
\texttt{diego.izquierdo@ens.fr}
}
\date{}
\normalsize
\maketitle

\small
\textbf{Résumé.} En 1986, Kato et Kuzumaki ont formulé des conjectures cherchant à donner une caractérisation diophantienne de la dimension cohomologique des corps. Dans cet article, nous montrons d'abord un énoncé de type principe local-global dans ce contexte pour les corps de nombres. Cela nous permet de donner une nouvelle démonstration d'une des conjectures de Kato et Kuzumaki pour les corps de nombres totalement imaginaires (la première preuve étant due à Olivier Wittenberg). Nos arguments nous permettent aussi d'obtenir des résultats pour les corps globaux de caractéristique positive. Dans la suite de l'article, nous établissons toutes les conjectures de Kato et Kuzumaki pour les corps $\mathbb{C}(x_1,...,x_n)$ et $\mathbb{C}(x_1,...,x_n)((t))$. Nous montrons finalement un résultat partiel pour le corps de séries de Laurent à deux variables $\mathbb{C}((x,y))$. \\

\textbf{Abstract.} In 1986, Kato and Kuzumaki stated several conjectures in order to give a diophantine characterization of cohomological dimension of fields. In this article, we first prove a local-global principle in this context for number fields. This allows us to give a new proof of one of Kato and Kuzumaki's conjectures for totally imaginary number fields (the first proof was given by Olivier Wittenberg). Our arguments can be generalized to get results for global fields of positive characteristic. We then establish all the conjectures for the fields $\mathbb{C}(x_1,...,x_n)$ and $\mathbb{C}(x_1,...,x_n)((t))$. We finally prove a partial result for the field of Laurent series in two variables $\mathbb{C}((x,y))$.
\normalsize

\section*{Introduction}

\hspace{3ex} En 1986, dans l'article \cite{KK}, Kato et Kuzumaki émettent des conjectures dont le but est de donner une caractérisation diophantienne de la dimension cohomologique des corps. Pour ce faire, ils introduisent des variantes des propriétés $C_i$ des corps qui concernent la $K$-théorie de Milnor et les hypersurfaces projectives de petit degré, et espèrent qu'elles caractérisent les corps de petite dimension cohomologique.\\

\hspace{3ex} Plus précisément, donnons-nous un corps $L$ et deux entiers naturels $q$ et $i$. On note $K_q^M(L)$ le $q$-ième groupe de $K$-théorie de Milnor de $L$. Pour chaque extension finie $L'$ de $L$, on dispose d'un morphisme norme $N_{L'/L}: K_q^M(L')\rightarrow K_q^M(L)$ (section 1.7 de \cite{Kat}). Ainsi, lorsque $Z$ est un schéma de type fini sur $L$, on peut introduire le sous-groupe $N_q(Z/L)$ de $K_q^M(L)$ engendré par les images des applications norme $N_{L'/L}$ lorsque $L'$ décrit les extensions finies de $L$ telles que $Z(L') \neq \emptyset$. On dit alors que le corps $L$ est $C_i^q$ si, pour tout $n\geq 1$, pour toute extension finie $L'$ de $L$ et pour toute hypersurface $Z$ de $\mathbb{P}^n_{L'}$ de degré $d$ avec $d^i\leq n$, on a $N_q(Z/L') = K_q^M(L')$. Par exemple, le corps $L$ est $C_i^0$ si, pour toute extension $L'$ de $L$, toute hypersurface $Z$ de $\mathbb{P}^n_{L'}$ de degré $d$ avec $d^i\leq n$ a un 0-cycle de degré 1. Le corps $L$ est $C_0^q$ si, pour toute tour d'extensions finies $L''/L'/L$, l'application norme $N_{L''/L'}: K_q^M(L'')\rightarrow K_q^M(L')$ est surjective.\\

\hspace{3ex} Kato et Kuzumaki conjecturent que, pour $i \geq 0$ et $q\geq 0$, un corps parfait est $C_i^q$ si, et seulement si, il est de dimension cohomologique au plus $i+q$. Comme cela a déjà été remarqué à la fin de l'article \cite{KK}, la conjecture de Bloch-Kato (qui a été établie par Rost et Voevodsky) montre qu'un corps est $C_0^q$ si, et seulement si, il est de dimension cohomologique au plus $q$. Par contre, la conjecture de Kato et Kuzumaki est fausse en toute généralité. Par exemple, Merkurjev exhibe dans \cite{Mer} un corps de caractéristique nulle de dimension cohomologique $2$ ne vérifiant pas la propriété $C^0_2$. De même, Colliot-Thélène et Madore exhibent dans \cite{CTM} un corps de caractéristique nulle de dimension cohomologique $1$ ne vérifiant pas la propriété $C^0_1$. Ces contre-exemples sont tous construits par une méthode de récurrence transfinie due à Merkurjev et Suslin.  La conjecture de Kato et Kuzumaki reste donc entièrement ouverte pour les corps qui interviennent usuellement en arithmétique ou en géométrie algébrique.\\

\hspace{3ex} Ce n'est que très récemment que cette question a connu des avancées importantes. En effet, c'est en 2015 que Wittenberg a démontré dans \cite{Wit} que les corps $p$-adiques, le corps $\mathbb{C}((t_1))((t_2))$ et les corps de nombres totalement imaginaires vérifient tous la propriété $C_1^1$. Sa méthode consiste à introduire et démontrer pour ces corps une propriété plus forte que la propriété $C_1^1$: plus précisément, il dit qu'un corps $L$ est $C^1_1$-fort si, pour toute extension finie $L'$ de $L$, tout $L'$-schéma propre $Z$ et tout faisceau cohérent $E$ sur $Z$, la caractéristique d'Euler-Poincaré $\chi(X,E)$ annule le groupe abélien $K_q^M(L')/N_q(Z/L')$. Il se trouve que cette notion se prête beaucoup mieux aux dévissages que la propriété $C_1^1$ de Kato et Kuzumaki: cela permet à Wittenberg d'employer les méthodes qui avaient été développées antérieurement dans \cite{ELW}. À la fin de son article, Wittenberg se demande si le corps des séries de Laurent à deux variables $\mathbb{C}((x,y))$ et le corps de fractions rationnelles $\mathbb{C}(x,y)$ vérifient aussi la propriété $C_1^1$: avec $\mathbb{C}(x)((y))$ et $\mathbb{C}((x))(y)$, ce sont les cas les plus simples de corps de dimension cohomologique 2 pour lesquels cette propriété n'est pas connue.\\

\hspace{3ex} Le présent article est structuré en six parties pouvant être lues de manière très indépendante et qui portent sur les conjectures de Kato et Kuzumaki pour différents corps. Dans la première section, nous nous penchons sur les corps de nombres, sur lesquels nous établissons un énoncé de type principe local-global qui n'était connu auparavant que pour les variétés projectives, lisses, géométriquement irréductibles (voir le théorème 4 de \cite{KS}) ou pour les variétés propres de caractéristique d'Euler-Poincaré égale à 1 (proposition 6.2 de \cite{Wit}):

\begin{thmx} \emph{(Théorème \ref{locglob})}\\\label{A}
Soient $K$ un corps de nombres et $\Omega_K$ l'ensemble des places de $K$. Soit $Z$ un $K$-schéma propre contenant un fermé géométriquement intègre. Pour $v\in \Omega_K$, on note $K_v$ le complété de $K$ en $v$ et $Z_v$ le $K_v$-schéma $Z \times_K K_v$. Alors:
$$\text{Ker}\left( K^{\times}/N_1(Z/K) \rightarrow \prod_{v \in \Omega_K} K_v^{\times}/N_1(Z_v/K_v) \right)=0.$$
\end{thmx}

Ce théorème nous permet ensuite de déduire une nouvelle preuve, différente de celle de \cite{Wit}, du fait que les corps de nombres totalement imaginaires vérifient la propriété $C_1^1$. S'il est vrai que cette preuve utilise la propriété $C_1^1$ des corps $p$-adiques (corollaire 5.5 de \cite{Wit}), le passage des résultats locaux aux résultats globaux est nettement plus simple et explicite que dans la section 6 de \cite{Wit}. \\

\hspace{3ex} La deuxième section est consacrée aux corps globaux de caractéristique positive $p>0$. Nous démontrons sur de tels corps un principe local-global similaire à celui du théorème \ref{A}:

\begin{thmx} \emph{(Théorème \ref{carp})}\\\label{B}
Soient $K$ le corps des fonctions d'une courbe sur un corps fini de caractéristique $p>0$ et $\Omega_K$ l'ensemble des places de $K$. Soit $Z$ un $K$-schéma propre contenant un fermé géométriquement irréductible. Pour $v\in \Omega_K$, on note $K_v$ le complété de $K$ en $v$ et $Z_v$ le $K_v$-schéma $Z \times_K K_v$. On note $N_1^s(Z/K)$ le sous-groupe de $K^{\times}$ engendré par les images des normes $N_{L_s/K}: L_s^{\times} \rightarrow K^{\times}$ où $L$ décrit les extensions finies de $K$ telles que $Z(L) \neq \emptyset$ et $L_s$ désigne la clôture séparable de $K$ dans $L$. Alors:
$$\text{Ker}\left( K^{\times}/N_1^s(Z/K) \rightarrow \prod_{v \in \Omega_K} K_v^{\times}/N_1^s(Z_v/K_v) \right)=0.$$
\end{thmx}

Comme pour les corps de nombres totalement imaginaires, cela nous permet de déduire ensuite une preuve de la propriété $C_1^1$ « hors de $p$ » pour les corps globaux de caractéristique positive.\\

\hspace{3ex} Dans la troisième partie, nous établissons les conjectures de Kato et Kuzumaki pour les corps de fonctions de variétés sur un corps algébriquement clos de caractéristique nulle:

\begin{thmx} \emph{(Théorème \ref{tresglobal})}\\\label{C}
Soit $k$ un corps algébriquement clos de caractéristique 0. Alors le corps $k(t_1,...,t_n)$ est un corps $C_i^q$ pour tous $i\geq 0$ et $q \geq 0$ tels que $i+q=n$.
\end{thmx}

En particulier, cela établit que le corps $\mathbb{C}(x,y)$ est $C_1^1$, ce qui répond affirmativement à la question (3) du paragraphe 7.3 de \cite{Wit}. Un tel résultat ne peut pas être obtenu par les méthodes développées par Wittenberg puisque le corps $\mathbb{C}(x,y)$ ne vérifie pas la propriété $C^1_1$-forte (remarque 7.6 de \cite{Wit}).\\

\hspace{3ex} Dans la quatrième partie, nous démontrons les propriétés de Kato et Kuzumaki pour les corps de valuation discrète complets dont le corps résiduel est le corps des fonctions d'une variété sur un corps algébriquement clos de caractéristique nulle:

\begin{thmx} \emph{(Théorème \ref{c(x)((t))})}\\\label{D}
Soit $k$ un corps algébriquement clos de caractéristique nulle. Alors le corps complet  $k(t_1,...,t_{n-1})((t))$ est un corps $C_i^q$ pour tous $i\geq 0$ et $q \geq 0$ tels que $i+q=n$.
\end{thmx}

En particulier, cela prouve que $\mathbb{C}(x)((t))$ est $C_1^1$. Encore une fois, un tel résultat ne peut pas être obtenu par les méthodes développées par Wittenberg puisque le corps $\mathbb{C}(x)((t))$ ne vérifie pas la propriété $C^1_1$-forte (remarque 7.6 de \cite{Wit}). \\

\hspace{3ex} Dans la cinquième partie, nous étudions les corps $\mathbb{C}((t_1))...((t_{m-1}))((x,y))$ et $\mathbb{C}((t_1))...((t_{m}))(x)$ et nous démontrons le théorème suivant:

\begin{thmx} \emph{(Théorème \ref{diff})}\\\label{E}
Soient $k$ un corps algébriquement clos de caractéristique nulle et $m, i, q$ des entiers naturels tels que $i+q=m+1$. Soient $L$ le corps des fonctions d'une courbe projective lisse géométriquement intègre sur $k((t_1))...((t_m))$ ou une extension finie de $k((t_1))...((t_{m-1}))((x,y))$. Soit $Z$ une hypersurface de degré $d$ dans $\mathbb{P}^n_L$ avec $d^i \leq n$. On suppose $d$ premier. Alors:
\begin{itemize}
\item[(i)] Il existe des extensions finies $L_1$, ..., $L_r$ de $L$ telles que $Z(L_s)\neq \emptyset$ pour chaque $s$ et le groupe $N_q(Z/L)$ est engendré par les sous-groupes $N_{L_i/L}(K_q^M(L_s))$ pour $s = 1,...,r$.
\item[(ii)] Le groupe $K_q^M(L)/N_q(Z/L)$ est fini.
\end{itemize}
\end{thmx}

En particulier, le théorème s'applique à $\mathbb{C}((t))(x)$ et à $\mathbb{C}((x,y))$ en prenant $m=1$ et $i=q=1$. Sa preuve utilise notamment certains résultats de \cite{Izq1} et \cite{Izq2}. On ne sait toujours pas si le corps $L$ du théorème est $C_i^q$.\\

\hspace{3ex} Dans la sixième et dernière partie, nous posons un certain nombre de questions et proposons plusieurs pistes pour aller au-delà du présent article.

\subsection*{Rappels}

\hspace{3ex} Soient $L$ un corps quelconque et $q\geq 0$ un entier. Le $q$-ième groupe de $K$-théorie de Milnor de $L$ est par définition le groupe $K_0^M(K)=\mathbb{Z}$ si $q =0$ et:
$$K_q^M(L):= \underbrace{L^{\times} \otimes_{\mathbb{Z}} ... \otimes_{\mathbb{Z}} L^{\times}}_{\text{$q$ fois}} / \left\langle x_1 \otimes ... \otimes x_q | \exists i,j, i\neq j, x_i+x_j=1 \right\rangle$$
si $q>0$. Pour $x_1,...,x_q \in L^{\times}$, le symbole $\{x_1,...,x_q\}$ désigne la classe de $x_1 \otimes ... \otimes x_q$ dans $K_q^M(L)$. Plus généralement, pour $r$ et $s$ des entiers naturels tels que $r+s=q$, on dispose d'un accouplement naturel:
$$K_r^M(L) \times K_s^M(L) \rightarrow K_q^M(L)$$
que l'on notera $\{\cdot, \cdot\}$.\\

\hspace{3ex} Lorsque $L'$ est une extension finie de $L$, on dispose d'une application norme $N_{L'/L}: K_q^M(L') \rightarrow K_q^M(L)$ (section 1.7 de \cite{Kat}) vérifiant les propriétés suivantes:
\begin{itemize}
\item[$\bullet$] pour $q=0$, l'application $N_{L'/L}: K_0^M(L') \rightarrow K_0^M(L)$ est la multiplication par $[L':L]$;
\item[$\bullet$] pour $q=1$, l'application $N_{L'/L}: K_1^M(L') \rightarrow K_1^M(L)$ coïncide avec la norme usuelle $L'^{\times} \rightarrow L^{\times}$;
\item[$\bullet$] si $r$ et $s$ sont des entiers naturels tels que $r+s=q$, on a $N_{L'/L}(\{x,y\})=\{x,N_{L'/L}(y)\}$ pour $x \in K_q^M(L)$ et $y\in K_q^M(L')$;
\item[$\bullet$] si $L''$ est une extension finie de $L'$, on a $N_{L''/L} = N_{L'/L} \circ N_{L''/L'}$.
\end{itemize}
\vspace{13pt}
\hspace{3ex} Pour chaque $L$-schéma de type fini, on note $N_q(Z/L)$ le sous-groupe de $K_q^M(L)$ engendré par les images de $N_{L'/L}: K_q^M(L') \rightarrow K_q^M(L)$ pour $L'$ décrivant les extensions finies de $L$ telles que $Z(L')\neq \emptyset$. En particulier, $N_0(Z/L)$ est le sous-groupe de $\mathbb{Z}$ engendré par l'indice de $Z$ (ie le pgcd des degrés $[L':L]$ pour $L'$ extension finie de $L$ telle que $Z(L') \neq \emptyset$). Pour $i\geq 0$, on dit que $L$ est un corps $C_i^q$ si, pour toute extension finie $L'$ de $L$ et pour toute hypersurface $Z$ de $\mathbb{P}^n_{L'}$ de degré $d$ avec $d^i \leq n$, on a $N_q(Z/L')=K_q^M(L')$. En particulier, $L$ est $C_i^0$ si, pour toute extension finie $L'$ de $L$, toute hypersurface $Z$ de $\mathbb{P}^n_{L'}$ de degré $d$ avec $d^i \leq n$ est d'indice 1. \\

\hspace{3ex} Le corps $L$ est $C_0^q$ si, pour toute tour d'extensions finies $L''/L'/L$, la norme $N_{L''/L'}: K_q^M(L'')\rightarrow K_q^M(L')$ est surjective. Comme cela a été remarqué par Kato et Kuzumaki à la fin de \cite{KK}, en tenant compte de la conjecture de Bloch-Kato qui permet d'identifier les groupes $K_q^M(L)/n$ et $ H^q(L,\mu_n^{\otimes q})$ pour $n$ premier à la caractéristique de $L$ et qui a été prouvée par Rost et Voevodsky, on peut montrer qu'un corps de caractéristique nulle est $C_0^q$ si, et seulement si, il est de dimension cohomologique au plus $q$.

\subsection*{Remerciements}

\hspace{3ex} Je tiens à remercier très chaleureusement David Harari et Olivier Wittenberg pour des discussions très intéressantes et enrichissantes, ainsi que pour de nombreuses suggestions ayant permis d'améliorer ce texte. Je voudrais aussi remercier Yong Hu pour plusieurs échanges intéressants. Je voudrais finalement remercier l'École Normale Supérieure pour les excellentes conditions de travail.

\section{Corps de nombres}\label{s1}

\hspace{3ex} Le résultat principal de cette section est le théorème suivant:

\begin{theorem}\label{locglob}
Soient $K$ un corps de nombres et $\Omega_K$ l'ensemble des places de $K$. Soit $Z$ un $K$-schéma propre contenant un fermé géométriquement intègre. Pour $v\in \Omega_K$, on note $K_v$ le complété de $K$ en $v$ et $Z_v$ le $K_v$-schéma $Z \times_K K_v$. Alors:
$$\text{Ker}\left( K^{\times}/N_1(Z/K) \rightarrow \prod_{v \in \Omega_K} K_v^{\times}/N_1(Z_v/K_v) \right)=0.$$
\end{theorem}

\begin{notation}
Lorsque $M$ est un module galoisien sur $K$, on notera:
$$\Sha^1(K,M) := \text{Ker} \left( H^1(K,M) \rightarrow \prod_{v\in \Omega_K} H^1(K_v,M) \right) .$$
\end{notation}

\begin{proof} Dans la suite, nous fixons une clôture algébrique $\overline{K}$ de $K$: toutes les extensions finies de $K$ seront vues à l'intérieur de $\overline{K}$.\\
Si $Z(K) \neq \emptyset$, le résultat est évident. Nous supposons désormais que $Z(K) = \emptyset$, et nous fixons $x \in K^{\times}$ dont la classe modulo $N_1(Z/K)$ est dans
$$\text{Ker}\left( K^{\times}/N_1(Z/K) \rightarrow \prod_{v \in \Omega_K} K_v^{\times}/N_1(Z_v/K_v) \right).$$
Nous voulons montrer que $x \in N_1(Z/K)$. Pour ce faire, considérons une extension finie galoisienne $L$ de $K$ telle que $Z(L) \neq \emptyset$. Soit $S \subseteq \Omega_K$ l'ensemble des places $v$ de $K$ vérifiant l'une des affirmations suivantes:
\begin{itemize}
\item[(i)] $v$ est finie et l'extension $L/K$ est ramifiée en $v$;
\item[(ii)] $v$ est finie et $x$ n'est pas une unité dans $\mathcal{O}_v$;
\item[(iii)] $v$ est infinie.
\end{itemize}
Bien sûr, $S$ est une partie finie de $\Omega_K$. Soit $v \in \Omega_K$. Deux cas se présentent:
\begin{itemize}
\item[$\bullet$] Supposons d'abord que $v \in \Omega_K \setminus S$.  Dans ce cas, $v$ est une place finie, et comme l'extension $L_v/K_v$ est non ramifiée, on sait que $N_{L_v/K_v}(L_v^{\times})$ contient $\mathcal{O}_v^{\times}$. De plus, $x\in \mathcal{O}_v^{\times}$. On en déduit que $x\in N_{L_v/K_v}(L_v^{\times})$. 
\item[$\bullet$] Supposons maintenant que $v \in S$ et fixons une clôture algébrique $\overline{K_v}$ de $K_v$. On sait que $x \in N_1(Z_v/K_v)$ par hypothèse. Soient alors $M_1^{(v)},...,M_{n_v}^{(v)}$ des extensions finies de $K_v$ contenues dans $\overline{K_v}$ telles que $x \in \prod_{i=1}^{n_v} N_{M_i^{(v)}/K_v}({M_i^{(v)}}^{\times})$ et $Z(M_i^{(v)}) \neq \emptyset$ pour chaque $i$. D'après le théorème d'approximation de Greenberg, on sait que $Z(M_i^{(v)} \cap \overline{K}) \neq \emptyset$. On note alors $L_i^{(v)}$ une extension finie de $K$ contenue dans $M_i^{(v)} \cap \overline{K}$ telle que $Z(L_i^{(v)}) \neq \emptyset$. On a bien sûr:
$$x \in \prod_{i=1}^{n_v} N_{L_i^{(v)} \otimes_{K} K_v/K_v}((L_i^{(v)} \otimes_{K} K_v)^{\times}).$$
\end{itemize}
Nous avons donc montré que, si $T$ est le tore normique $R^{1}_{E/K}(\mathbb{G}_m)$ avec $E=L \times \prod_{v\in S}\prod_{i=1}^{n_v} L_i^{(v)}$ et si $[x]$ désigne l'image de $x$ dans $$H^1(K,T) = K^{\times}/ N_{E/K}(E^{\times}),$$ alors:
\begin{equation}\label{un}
[x] \in \Sha^1(K,T).
\end{equation}

Considérons maintenant la plus petite extension finie galoisienne $F$ de $K$ contenant $L$ et tous les $L_i^{(v)}$, et montrons que $Z$ possède un point dans une extension finie $F_0$ de $K$ telle que $F_0 \cap F =K$.\\
 
Soit $Y$ une sous-variété géométriquement intègre de $Z$. Comme $Z$ n'a pas de points rationnels, la dimension de $Y$ est au moins 1. Le théorème de Bertini appliqué à un ouvert dense et quasi-projectif de $Y$ implique alors que $Y$ contient une courbe $C$ quasi-projective géométriquement intègre sur $K$. Soient $C'$ une courbe de $\mathbb{P}^2_K$ birationnellement équivalente à $C$ et $g \in K[X,Y,Z]$ un polynôme homogène (absolument irréductible) tel que $C'$ est la courbe d'équation $g=0$. Soit $U'$ un ouvert non vide de $C'$ isomorphe à un ouvert de $C$. Supposons sans perte de généralité que $g \not\in K[Y,Z]$ et considérons l'ensemble:
$$H:= \{(y,z) \in F^2 | g(X,y,z)\in F[X] \text{ est irréductible}\}.$$
C'est un sous-ensemble Hilbertien de $F^2$ car $g$ est irréductible sur $F$. D'après le corollaire 12.2.3 de \cite{FA}, $H$ contient un sous-ensemble Hilbertien $H'$ de $K^2$. Comme $K$ est un corps Hilbertien, $H'$ est infini, et donc il existe une infinité de couples $(y,z) \in K^2$ tels que $g(X,y,z)$ est irréductible sur $F$. Chacun de ces couples correspond à un point $y \in (C')^{(1)}$ tel que $K(y) \cap F = K$. Comme $C' \setminus U'$ est fini, on déduit qu'il existe $y \in (U')^{(1)}$ tel que $K(y)$ est une extension finie de $K$ vérifiant $K(y) \cap F =K$. En posant $F_0 = K(y)$, on obtient que $Z(F_0)\neq \emptyset$ et $F_0 \cap F =K$.\\

D'après le théorème 1 de \cite{DW}, on a:
$$\Sha^1(K,Q) = 0$$
où $Q$ est le tore normique $R^{1}_{E'/K}(\mathbb{G}_m)$, avec $E'=L \times F_0 \times \prod_{v\in S}\prod_{i=1}^{n_v} L_i^{(v)}$. En combinant cela avec (\ref{un}), on voit que la classe de $x$ dans $H^1(K,Q)$ est triviale, et donc que $x \in N_1(Z/K)$.
\end{proof}

\begin{corollary}\label{prem}
Soient $K$ un corps de nombres et $\Omega_K$ l'ensemble des places de $K$. Soit $Z$ un $K$-schéma propre. Pour $v\in \Omega_K$, on note $K_v$ le complété de $K$ en $v$ et $Z_v$ le $K_v$-schéma $Z \times_K K_v$. On suppose qu'il existe des extensions finies $K_1,...,K_r$ de $K$ telles que $Z_{K_i}$ contient un fermé géométriquement intègre pour chaque $i$ et les degrés $[K_i:K]$ sont premiers entre eux dans leur ensemble. Alors:
$$\text{Ker}\left( K^{\times}/N_1(Z/K) \rightarrow \prod_{v \in \Omega_K} K_v^{\times}/N_1(Z_v/K_v) \right)=0.$$
\end{corollary}

\begin{remarque}
Ce corollaire n'était connu auparavant que pour les $K$-variétés projectives, lisses, géométriquement irréductibles (théorème 4 de \cite{KS}) et pour les variétés propres de caractéristique d'Euler-Poincaré égale à 1 (proposition 6.2 de \cite{Wit}). Il généralise ces résultats d'après la proposition 3.3 de \cite{Wit}.
\end{remarque}

\begin{proof}
D'après le théorème \ref{locglob}, pour chaque $i$, on a:
$$\text{Ker}\left( K_i^{\times}/N_1(Z_{K_i}/K_i) \rightarrow \prod_{w \in \Omega_{K_i}} K_{i,w}^{\times}/N_1(Z_{K_{i,w}}/K_{i,w}) \right)=0.$$
Un argument de restriction-corestriction montre alors que le groupe $$\text{Ker}\left( K^{\times}/N_1(Z/K) \rightarrow \prod_{v \in \Omega_K} K_v^{\times}/N_1(Z_v/K_v) \right)$$
est de $[K_i:K]$-torsion pour chaque $i$. Il est donc nul.
\end{proof}

\hspace{3ex} Wittenberg a démontré assez récemment la propriété $C_1^1$ pour les corps de nombres totalement imaginaires (théorème 6.1 de \cite{Wit}). Le théorème \ref{locglob} nous permet de retrouver ce résultat par une méthode très différente. Le passage des résultats locaux aux résultats globaux est beaucoup plus explicite que dans la section 6 de \cite{Wit}.

\begin{corollary}\label{aze}
Soient $K$ un corps de nombres et $Z$ une hypersurface de degré $d$ de $\mathbb{P}^n_K$ telle que $d \leq n$ et $N_1(Z_v/K_v)=K_v^{\times}$ pour toute place réelle $v$ de $K$. Alors $N_1(Z/K)=K^{\times}$.
\end{corollary}

\begin{proof}
On sait que $\chi(Z,\mathcal{O}_Z)=1$. La proposition 3.3 de \cite{Wit} montre alors que les hypothèses du corollaire \ref{prem} sont vérifiées. On en déduit que:
$$\text{Ker}\left( K^{\times}/N_1(Z/K) \rightarrow \prod_{v \in \Omega_K} K_v^{\times}/N_1(Z_v/K_v) \right)=0.$$
Or d'après le corollaire 5.5 de \cite{Wit}, on a $N_1(Z_v/K_v)=K_v^{\times}$ pour chaque place finie $v$ de $K$. Par hypothèse, on a aussi $N_1(Z_v/K_v)=K_v^{\times}$ pour chaque place infinie $v$ de $K$. On en déduit que $N_1(Z/K)=K^{\times}$.
\end{proof}

\begin{remarque}
Au lieu d'utiliser la proposition 3.3 de \cite{Wit} et le corollaire \ref{prem} pour montrer que
$$\text{Ker}\left( K^{\times}/N_1(Z/K) \rightarrow \prod_{v \in \Omega_K} K_v^{\times}/N_1(Z_v/K_v) \right)=0,$$
on aurait pu combiner le théorème 2 de \cite{Kol} et le théorème \ref{locglob}. La preuve de la proposition 3.3 de \cite{Wit} est néanmoins beaucoup plus élémentaire que  celle du théorème 2 de \cite{Kol}.
\end{remarque}

\section{Corps globaux de caractéristique positive}\label{s2}

\hspace{3ex} Dans cette section, nous nous intéressons aux corps globaux de caractéristique positive. Dans ce contexte, nous définissons des variantes du groupe $N_1(Z/K)$:

\begin{definition}
Soit $K$ un corps de caractéristique $p>0$. Soit $Z$ un $K$-schéma.
\begin{itemize}
\item[(i)] On note $N_1^p(Z/K)$ l'ensemble des $x \in K^{\times}$ tels qu'il existe un entier $r\geq 1$ vérifiant $x^{p^r}\in N_1(Z/K)$.
\item[(ii)] On note $N_1^s(Z/K)$ le sous-groupe de $K^{\times}$ engendré par les images de normes $N_{L_s/K}: L_s^{\times} \rightarrow K^{\times}$ où $L$ décrit les extensions finies de $K$ telles que $Z(L) \neq \emptyset$ et $L_s$ désigne la clôture séparable de $K$ dans $L$.
\end{itemize}
\end{definition}

Nous pouvons maintenant établir le pendant du théorème \ref{locglob} pour les corps globaux de caractéristique positive:

\begin{theorem}\label{carp}
Soient $K$ le corps des fonctions d'une courbe sur un corps fini de caractéristique $p>0$ et $\Omega_K$ l'ensemble des places de $K$. Soit $Z$ un $K$-schéma propre contenant un fermé géométriquement irréductible. Pour $v\in \Omega_K$, on note $K_v$ le complété de $K$ en $v$ et $Z_v$ le $K_v$-schéma $Z \times_K K_v$.
\begin{itemize}
\item[(i)] On a:
$$\text{Ker}\left( K^{\times}/N_1^s(Z/K) \rightarrow \prod_{v \in \Omega_K} K_v^{\times}/N_1^s(Z_v/K_v) \right)=0.$$
\item[(ii)] Le groupe:
$$\text{Ker}\left( K^{\times}/N_1(Z/K) \rightarrow \prod_{v \in \Omega_K} K_v^{\times}/N_1^p(Z_v/K_v) \right)$$
est de torsion $p$-primaire.
\end{itemize}
\end{theorem}

\begin{notation}
On notera $K^s$ la clôture séparable de $K$. Lorsque $M$ est un module galoisien sur $K$, on notera:
$$\Sha^1(K,M) := \text{Ker} \left( H^1(K,M) \rightarrow \prod_{v\in \Omega_K} H^1(K_v,M) \right) .$$
\end{notation}

\begin{proof}
Montrons d'abord (i). S'il existe une extension finie purement inséparable $K'$ de $K$ telle que $Z(K') \neq \emptyset$, le résultat est évident. Nous supposons désormais que $Z(K') = \emptyset$ pour toute extension finie purement inséparable $K'$ de $K$, et nous fixons $x \in K^{\times}$ dont la classe modulo $N_1^s(Z/K)$ est dans
$$\text{Ker}\left( K^{\times}/N_1^s(Z/K) \rightarrow \prod_{v \in \Omega_K} K_v^{\times}/N_1^s(Z_v/K_v) \right).$$
Nous voulons montrer que $x \in N_1^s(Z/K)$. Pour ce faire, considérons une extension finie normale $L$ de $K$ telle que $Z(L) \neq \emptyset$. Soit $L_s$ la clôture séparable de $K$ dans $L$: c'est une extension finie galoisienne de $K$. Soit $S \subseteq \Omega_K$ l'ensemble des places $v$ de $K$ vérifiant l'une des affirmations suivantes:
\begin{itemize}
\item[(i)] l'extension $L_s/K$ est ramifiée en $v$.
\item[(ii)] $x$ n'est pas une unité dans $\mathcal{O}_v$.
\end{itemize}
Bien sûr, $S$ est une partie finie de $\Omega_K$. Soit $v \in \Omega_K$. Deux cas se présentent:

\begin{itemize}
\item[$\bullet$] Supposons d'abord que $v \in \Omega_K \setminus S$.  Comme l'extension ${L_s}_v/K_v$ est non ramifiée, on sait que $N_{{L_s}_v/K_v}({L_s}_v^{\times})$ contient $\mathcal{O}_v^{\times}$. De plus, $x\in \mathcal{O}_v^{\times}$. On en déduit que $x\in N_{{L_s}_v/K_v}(L_v^{\times})$.
\item[$\bullet$] 
Supposons maintenant que $v \in S$ et fixons une clôture algébrique $\overline{K_v}$ de $K_v$. On sait que $x \in N_1^s(Z_v/K_v)$ par hypothèse. Soient alors $M_1^{(v)},...,M_{n_v}^{(v)}$ des extensions finies de $K_v$ contenues dans $\overline{K_v}$ telles que, si $M_{i,s}^{(v)}$ désigne la clôture séparable de $K_v$ dans $M_i^{(v)}$, on a $x \in \prod_{i=1}^{n_v} N_{M_{i,s}^{(v)}/K_v}({M_{i,s}^{(v)}}^{\times})$ et $Z(M_i^{(v)}) \neq \emptyset$ pour chaque $i$. D'après le théorème d'approximation de Greenberg, on sait que $Z(M_i^{(v)} \cap \overline{K}) \neq \emptyset$. On note alors $L_i^{(v)}$ une extension finie de $K$ contenue dans $M_i^{(v)} \cap \overline{K}$ telle que $Z(L_i^{(v)}) \neq \emptyset$. Soit $L_{i,s}^{(v)}$ la clôture séparable de $K$ dans $L_i^{(v)}$. On a bien sûr:
$$x \in \prod_{i=1}^{n_v} N_{L_{i,s}^{(v)} \otimes_{K} K_v/K_v}((L_{i,s}^{(v)} \otimes_{K} K_v)^{\times}).$$
\end{itemize}

Nous avons donc montré que, si $T$ est le tore normique $R^{1}_{E/K}(\mathbb{G}_m)$ avec $E=L_s \times \prod_{v\in S}\prod_{i=1}^{n_v} L_{i,s}^{(v)}$ et si $[x]$ désigne l'image de $x$ dans $H^1(K,T)$, alors:
\begin{equation}\label{deux}
[x] \in \Sha^1(K,T).
\end{equation}

Considérons maintenant la plus petite extension finie galoisienne $F$ de $K$ contenant $L_s$ et tous les $L_{i,s}^{(v)}$, et montrons que $Z$ possède un point dans une extension finie $F_0$ de $K$ telle que $F_0 \cap F =K$.\\

Soit $Y$ un fermé réduit géométriquement irréductible de $Z$. Comme $Z(K') = \emptyset$ pour toute extension finie purement inséparable $K'/K$, la dimension de $Y$ est au moins 1. Le théorème de Bertini appliqué à un ouvert dense et quasi-projectif de $Y$ implique alors que $Y$ contient une courbe $C$ quasi-projective géométriquement irréductible sur $K$. Soient $C'$ une courbe de $\mathbb{P}^2_K$ birationnellement équivalente à $C$ et $g \in K[X,Y,Z]$ un polynôme homogène (irréductible sur $K^s$) tel que $C'$ est la courbe d'équation $g=0$. Soit $U'$ un ouvert non vide de $C'$ isomorphe à un ouvert de $C$. Soient $K'$ une extension finie purement inséparable de $K$ et $g'\in K'[X,Y,Z] \setminus K'[X^p,Y^p,Z^p]$ irréductible sur $K'^s$ tels que $g=g'^{p^s}$. Sans perte de généralité, supposons que $g'\in K'[X,Y,Z] \setminus K'[X^p,Y,Z]$. Soient $m \geq 1$ un entier premier avec $p$ et $h'\in K'[Y,Z]\setminus \{0\}$ tels que le coefficient de $X^m$ dans $g'$ soit $h'$. Considérons l'ensemble:
$$H:= \{(y,z) \in F'^2 | g'(X,y,z)\in F'[X] \text{ est irréductible}, h'(y,z)\neq 0\},$$
où $F' = K' \cdot F$. C'est un sous-ensemble Hilbertien séparable de $F'^2$ car $g'$ est irréductible sur $F'$ et séparable en $X$. D'après le corollaire 12.2.3 de \cite{FA}, $H$ contient un sous-ensemble Hilbertien séparable $H'$ de $K'^2$. Comme $K'$ est un corps Hilbertien, $H'$ est infini, et donc il existe une infinité de couples $(y,z) \in K'^2$ tels que $g(X,y,z)$ est irréductible sur $F'$ et $h'(y,z) \neq 0$. Chacun de ces couples correspond à un point $y \in (C'_{K'})^{(1)}$ tel que $K'(y) \cap F' = K'$ et l'extension $K'(y)/K'$ est séparable. Comme $C'_{K'} \setminus U'_{K'}$ est fini, on déduit qu'il existe $y \in (U'_{K'})^{(1)}$ tel que $K'(y)$ est une extension finie séparable de $K'$ telle que $K'(y) \cap F' =K'$. En posant $F_0 = K'(y)$, on obtient que $Z(F_0)\neq \emptyset$ et $F_0 \cap F =K$.\\

Notons $F_{0,s}$ la clôture séparable de $K$ dans $F_0$. D'après le théorème 1 de \cite{DW}, on a:
$$\Sha^1(K,Q) = 0$$
où $Q$ est le tore normique $R^{(1)}_{E'/K}(\mathbb{G}_m)$, avec $E'=L_s \times F_{0,s} \times \prod_{v\in S}\prod_{i=1}^{n_v} L_{i,s}^{(v)}$. En combinant cela avec (\ref{deux}), on voit que la classe de $x$ dans $H^1(K,Q)$ est triviale, ce qui achève la preuve de (i). \\

Déduisons maintenant (ii) de (i). Soit $x\in K^{\times}$ dont la classe modulo $N_1(Z/K)$ est dans:
$$\text{Ker}\left( K^{\times}/N_1(Z/K) \rightarrow \prod_{v \in \Omega_K} K_v^{\times}/N_1^p(Z_v/K_v) \right).$$
Par hypothèse, pour chaque $v\in \Omega_K$, il existe $r_v \geq 0$ tel que $x^{p^{r_v}} \in N_1(Z_v/K_v) \subseteq N_1^s(Z_v/K_v)$. De plus, il existe un entier $m\geq 0$ tel que $x^m \in N_1^s(Z_v/K_v)$ pour tout $v\in \Omega_K$. On en déduit qu'il existe $r\geq 0$ tel que $x^{p^r} \in N_1^s(Z_v/K_v)$ pour tout $v\in\Omega_K$. D'après (i), cela prouve que $x^{p^r}\in N_1^s(Z/K)$. Considérons donc des extensions finies $K_1,...,K_n$ de $K$ telles que, si $K_{i,s}$ désigne la clôture séparable de $K$ dans $K_i$, on a $x^{p^r}\in \prod_{i=1}^n N_{K_{i,s}/K}(K_{i,s}^{\times})$ et $Z(K_i) \neq \emptyset$ pour chaque $i$. Comme les degrés $[K_i:K_{i,s}]$ sont tous des puissances de $p$, cela prouve qu'il existe $r' \geq 0$ tel que $(x^{p^r})^{p^{r'}} \in  \prod_{i=1}^n N_{K_{i}/K}(K_{i}^{\times})$. On en déduit que $x^{p^{r+r'}}\in N_1(Z/K)$, ce qui établit (ii).
\end{proof}

\hspace{3ex} Nous sommes maintenant en mesure d'établir la propriété $C_1^1$ « hors de $p$ » pour les corps globaux de caractéristique $p$.

\begin{corollary}
Soient $K$ le corps des fonctions d'une courbe sur un corps fini de caractéristique $p>0$ et $Z$ une hypersurface de degré $d$ de $\mathbb{P}^n_K$ telle que $d \leq n$. Alors le groupe $K^{\times}/N_1(Z/K)$ est d'exposant une puissance de $p$.
\end{corollary}

\begin{remarque}
Ce résultat n'est pas établi dans l'article \cite{Wit}. Il est possible que la méthode employée dans la preuve du théorème 6.1 dudit article s'adapte au cas des corps globaux de caractéristique positive en utilisant les résultats sur les altérations de de Jong. La preuve qui suit est probablement nettement plus simple.
\end{remarque}

\begin{proof}
Lorsque $A$ est un groupe abélien de torsion, on note $A\{p'\}$ le sous-groupe de $A$ constitué des éléments dont l'ordre est premier à $p$. Pour chaque $K$-schéma propre $Z$, on pose:
$$H_1(Z/K)=K^{\times}/N_1(Z/K)$$
et on note $n_Z$ l'exposant de $H_1(Z/K)\{p'\}$ si $Z$ est non vide ou 0 sinon. On dira que $Z$ possède la propriété P si $Z$ est normal. Nous allons vérifier les trois conditions de la proposition 2.1 de \cite{Wit}.

\begin{itemize}
\item[(1)] C'est évident, puisque qu'un morphisme de $K$-schémas propres $Y \rightarrow Z$ induit un morphisme surjectif $H_1(Y/K) \rightarrow H_1(Z/K)$.
\item[(2)] Soit $Z$ un $K$-schéma propre normal. Soit $K'$ la clôture algébrique de $K$ dans $K(Z)$. Alors $Z$ est muni d'une structure naturelle de $K'$-schéma propre géométriquement irréductible. D'après le théorème \ref{carp}: $$\text{Ker}\left( H_1(Z/K') \rightarrow \prod_{v \in \Omega_K} H_1(Z_v/K_v') \right)\{p'\}=0.$$ De plus, comme $K_v'$ est $C^1_1$-fort hors de $p$ pour chaque $v \in \Omega_K$ d'après le corollaire 4.7 de \cite{Wit}, le groupe $H_1(Z_v/K'_v)\{p'\}$ est tué par $\chi_{K'} (Z,\mathcal{O}_Z)$. On en déduit que $H_1(Z/K')\{p'\}$ est aussi tué par $\chi_{K'} (Z,\mathcal{O}_Z)$. Or $\chi_{K} (Z,\mathcal{O}_Z)=[K':K]\chi_{K'} (Z,\mathcal{O}_Z)$. Un argument de restriction-corestriction montre alors que $\chi_{K} (Z,\mathcal{O}_Z)$ tue $H_1(Z/K)\{p'\}$ et donc que $n_Z$ divise $\chi_{K} (Z,\mathcal{O}_Z)$.
\item[(3)] Il suffit de choisir le morphisme de normalisation.
\end{itemize}
La proposition 2.1 de \cite{Wit} montre alors que $K$ est $C^1_1$ hors de $p$.
\end{proof}

\section{Corps de fonctions de variétés sur un corps algébriquement clos}

\hspace{3ex} Dans cette section, nous allons établir les conjectures de Kato et Kuzumaki pour les corps de fonctions de variétés sur des corps algébriquement clos de caractéristique nulle. Nous avons déjà rappelé que la conjecture de Bloch-Kato implique qu'un corps de caractéristique 0 est $C_0^q$ si, et seulement si, il est de dimension cohomologique au plus $q$. La proposition qui suit est un cas particulier de ce résultat. Nous en donnons malgré tout une preuve élémentaire, dont les idées seront utiles par la suite pour démontrer les théorèmes \ref{tresglobal} et \ref{c(x)((t))}:

\begin{proposition}
Soit $k$ un corps algébriquement clos de caractéristique nulle. Alors le corps $K=k(t_1,...,t_q)$ est un corps $C_0^{q}$.
\end{proposition}

\begin{proof}
On procède par récurrence sur $q$. Le résultat est évident pour $q=0$. Supposons donc le résultat vrai pour un certain $q \geq 0$ et notons $K=k(t_1,...,t_{q+1})$. Soient $L_1$ une extension finie de $K$ et $L_2$ une extension finie de $L_1$. Soient $u_1,...,u_{q+1}$ des éléments de $L_1^{\times}$. Nous allons montrer que $\{u_1,...,u_{q+1}\} \in N_{L_2/L_1}(K_{q+1}^M(L_2))$.\\

Pour ce faire, construisons une famille $(w_1,...,w_s)$ d'éléments de $L_1^{\times}$ définie de la manière suivante:
\begin{itemize}
\item[$\bullet$] si $u_1,...,u_{q+1}$ ne sont pas algébriquement indépendants sur $k$, on se donne une base de transcendance $(v_1,...,v_r)$ de l'extension $L_1/k(u_1,...,u_{q+1})$ et on pose $(w_1,...,w_s)=(u_1,...,u_{q+1},v_1,...,v_{r-1})$. 
\item[$\bullet$] si $u_1,...,u_{q+1}$ sont algébriquement indépendants sur $k$, on pose $(w_1,...,w_s)=(u_1,...,u_q)$. 
\end{itemize}
Soit $M_1$ (resp. $M_2$) la clôture algébrique de $k(w_1,...,w_s)$ dans $L_1$ (resp. $L_2$). Soit $C_1$ (resp. $C_2$) une courbe géométriquement intègre sur $M_1$ (resp. $M_2$) telle que $M_1(C_1) = L_1$ (resp. $M_2(C_2) = L_2$). Comme $\overline{M_1}(C_1)$ est un corps $C_1$ et $\overline{M_2}(C_2)/\overline{M_1}(C_1)$ est une extension finie, on sait que $u_{q+1} \in N_{\overline{M_2}(C_2)/\overline{M_1}(C_1)}(\overline{M_2}(C_2)^{\times})$. Du coup, il existe une extension finie $F$ de $M_2$ et $y \in F(C_2)^{\times}$ tels que $u_{q+1} = N_{F(C_2)/F(C_1)}(y)$. De plus, par hypothèse de récurrence, $M_1$ est un corps $C_0^q$, et donc il existe $x \in K_{q}^M(F)$ tel que $\{u_1,...,u_{q}\} = N_{F/M_1}(x)$. Du coup, on a:
\begin{align*}
N_{L_2/L_1}(N_{F(C_2)/L_2}(\{x,y\}))&=N_{F(C_2)/M_1(C_1)}(\{x,y\})\\& = N_{F(C_1)/M_1(C_1)}(N_{F(C_2)/F(C_1)}(\{x,y\}))\\& = N_{F(C_1)/M_1(C_1)}(\{x,u_{q+1}\})\\&=\{u_1,...,u_{q+1}\}.
\end{align*}
On a donc montré que $\{u_1,...,u_{q+1}\} \in N_{L_2/L_1}(K_{q+1}^M(L_2))$, et donc le corps $K$ est bien $C_0^q$.
\end{proof}

\hspace{3ex} Nous pouvons maintenant établir les conjectures de Kato et Kuzumaki pour le corps des fonctions d'une variété sur un corps algébriquement clos de caractéristique nulle:

\begin{theorem} \label{tresglobal}
Soit $k$ un corps algébriquement clos de caractéristique nulle. Alors le corps $K=k(t_1,...,t_q)$ est un corps $C_i^{j}$ pour tous $i \geq 0$ et $j \geq 0$ tels que $i+j=q$.
\end{theorem}

\begin{proof}
Notons $K=k(t_1,...,t_{q})$. Soient $i \geq 0$ et $j \geq 0$ tels que $i+j=q$. Si $j=0$, le résultat est évident car $K$ est un corps $C_{q}$. Si $i=0$, le résultat découle de la proposition précédente. On suppose donc que $i \neq 0$ et $j \neq 0$. \\
Fixons maintenant une extension finie $L$ de $K$. Soient $Z$ une hypersurface de degré $d$ dans $\mathbb{P}^n_K$ avec $d^i \leq n$ et $u_1,...,u_j$ des éléments de $L^{\times}$. Soit $(v_1,...,v_r)$ une base de transcendance de l'extension $L/k(u_1,...,u_{j})$ (avec $r\geq 0$). Soient $M$ la clôture algébrique de $k(u_1,...,u_j,v_{q-j},...,v_{r})$ dans $L$ (de sorte que le degré de transcendance de $M/k$ soit $j$) et $X$ une $M$-variété géométriquement intègre sur $L$ de dimension $i$ telle que $M(X) = L$. Comme le corps $\overline{M}(X)$ est $C_i$, la variété $Z$ possède des points dans $\overline{M}(X)$. Par conséquent, il existe une extension finie $F$ de $M$ telle que $Z(F(X)) \neq \emptyset$. De plus, comme $M$ est $C_0^j$ d'après la proposition précédente, on sait que $\{u_1,...,u_j\} \in N_{F/M}(K_j^M(F))$. Par conséquent, $\{u_1,...,u_j\} \in N_{F(X)/M(X)}(K_j^M(F(X)))$, et $K$ est bien $C_i^j$.
\end{proof}

\begin{remarque}~
\begin{itemize}
\item[$\bullet$] Dans le théorème précédent, on a en fait montré que, si $L$ est une extension finie de $K$ et $Z$ est une hypersurface de degré $d$ dans $\mathbb{P}^n_L$ avec $d^i \leq n$, pour tout $j$-symbole $x \in K_j^M(L)$, il existe une extension finie $M$ de $L$ telle que $Z(M) \neq \emptyset$ et $x\in N_{M/L}(K_j^M(M))$. En particulier, si $i=q-1$ et $j=1$, pour tout élément $x$ de $L^{\times}$, il existe une extension finie $M$ de $L$ telle que $Z(M) \neq \emptyset$ et $x\in N_{M/L}(M^{\times})$.
\item[$\bullet$] Le théorème précédent ne peut pas être établi en utilisant les techniques de \cite{Wit}, puisque $\mathbb{C}(x,y)$ n'est pas $C^1_1$-fort.
\end{itemize}
\end{remarque}

\section{Corps locaux}

\subsection{Problème et stratégie}

\hspace{3ex} Le but de cette section est de démontrer les conjectures de Kato et Kuzumaki pour les corps de valuation discrète complets dont le corps résiduel est le corps des fonctions d'une variété sur un corps algébriquement clos de caractéristique 0. Cela concerne par exemple le corps $\mathbb{C}(x)((t))$, pour lequel on sait qu'il vérifie les propriétés $C_0^2$ et $C_2^0$ et pour lequel nous allons établir la propriété $C_1^1$.\\

\hspace{3ex} La principale difficulté pour démontrer que $K=\mathbb{C}(x)((t))$ est $C_1^1$ consiste à prouver que, si $Z$ est une hypersurface de $\mathbb{P}^n_{K}$ de degré $d\leq n$, alors $t\in N_1(Z/K)$. Pour ce faire, nous allons montrer que si l'on adjoint toutes les racines de $t$ à $K$, alors le corps obtenu $K_{\infty}$ est $C_1$: cela impliquera que $Z(K_{\infty}) \neq \emptyset$. Afin d'établir la propriété $C_1$ pour $K_{\infty}$, nous allons devoir établir un critère modulaire permettant de déterminer si une variété affine sur $K_{\infty}$ possède des points rationnels (corollaire \ref{infty}). À cet effet on utilisera fortement les constructions de l'article \cite{Gre} de Greenberg.

\subsection{Le théorème d'approximation de Greenberg revisité}

\hspace{3ex} Commençons par rappeler le théorème 1 de \cite{Gre}:

\begin{theorem} \textit{(Théorème 1 de \cite{Gre})}\\ \label{Green}
Soit $R$ un anneau de valuation discrète hensélien dont on note $K$ le corps des fractions. Soit $t$ une uniformisante de $R$. Soit $\underline{F}=(F_1,...,F_r)$ un système de $r$ polynômes à $n$ variables à coefficients dans $R$. On suppose $K$ de caractéristique 0. Il existe des entiers $N\geq 1$, $c\geq 1$ et $s\geq 0$ (dépendant uniquement de l'idéal $\underline{F}R[\underline{X}]$ de $R[\underline{X}]$ engendré par $F_1,...,F_r$) tels que, pour chaque $\nu \geq N$ et chaque $\underline{x} \in R^n$ vérifiant: $$\underline{F}(\underline{x})\equiv 0 \mod t^{\nu},$$ il existe $\underline{y}\in R^n$ tel que: $$\underline{y}\equiv \underline{x} \mod t^{[\nu/c]-s} \;\;\; \text{et} \;\;\; \underline{F}(\underline{y})=0.$$
\end{theorem}

En particulier, si le système $\underline{F}=0$ possède des solutions modulo $t^m$ pour chaque $m \geq 1$, alors il possède des solutions dans $R$. \\

\hspace{3ex} Fixons pour la suite un anneau de valuation discrète hensélien $R$ de corps des fractions $K$. Soit $t$ une uniformisante de $R$. On suppose $K$ de caractéristique 0. Pour $q\geq 1$, on pose $K_q=K(t^{1/q})$ et $R_q=\mathcal{O}_{K_q}$. On pose aussi $K_{\infty} = \bigcup_{q\geq 1} K_q$ et $R_{\infty}= \bigcup_{q\geq 1} R_q$. Nous voulons établir un résultat similaire au théorème \ref{Green} pour le corps $K_{\infty}$. Pour ce faire, nous commençons par établir un lemme simple d'algèbre commutative.

\begin{definition}
On dit qu'un idéal $I$ de $R[\underline{X}]$ est $t$-saturé si, pour tout $f\in R[\underline{X}]$ tel que $tf\in I$, on a $f\in I$.
\end{definition}

\begin{remarque}
Bien sûr, la définition précédente est indépendante du choix de l'uniformisante $t$. Mais, comme dans la suite nous allons devoir remplacer $R$ par $R_q$, il sera utile de se souvenir d'une uniformisante de l'anneau sur lequel on travaille.
\end{remarque}

\begin{lemma}\label{sature}
Soit $I$ un idéal de $R[\underline{X}]$.
\begin{itemize}
\item[(i)] Si $I$ est $t$-saturé, alors $IR_q[\underline{X}]$ est $t^{1/q}$-saturé pour tout $q\geq 1$.
\item[(ii)] Si $I$ est radical et $t$-saturé, alors $IR_q[\underline{X}]$ est radical pour tout $q \geq 1$.
\end{itemize}
\end{lemma}

\begin{proof}
\begin{itemize}
\item[(i)] Supposons que $I$ soit $t$-saturé. Fixons $q \geq 1$. Soit $f\in R_q[\underline{X}]$ tel que $t^{1/q}f \in IR_q[\underline{X}]$. Écrivons $t^{1/q}f = \sum_{i=1}^r f_ig_i$, avec $f_i \in I$ et $g_i \in R_q[\underline{X}]$. Pour chaque $i\in \{1,...,n\}$, considérons $h_i \in R[\underline{X}]$ tel que $t^{1/q}$ divise $g_i-h_i$ dans $R_q$ (ie la valuation de $g_i-h_i$ est strictement positive): cela est possible car $R$ et $R_q$ ont même corps résiduel. On écrit alors:
$$t^{1/q}f = \sum_{i=1}^r f_i(g_i-h_i) + \sum_{i=1}^r f_ih_i.$$
Du coup, $t$ divise $\sum_{i=1}^r f_ih_i$ dans $R[\underline{X}]$. Comme $I$ est $t$-saturé et $\sum_{i=1}^r f_ih_i\in I$, on déduit que $\frac{\sum_{i=1}^r f_ih_i}{t}\in I$. La relation:
$$f = \sum_{i=1}^r f_i\frac{g_i-h_i}{t^{1/q}} + t^{(q-1)/q}\cdot \frac{\sum_{i=1}^r f_ih_i}{t}$$
implique alors que $f\in IR_q[\underline{X}]$et donc que l'idéal $IR_q[\underline{X}]$ est bien $t^{1/q}$-saturé.
\item[(ii)] Supposons que $I$ soit radical et $t$-saturé. Fixons $q \geq 1$. Soit $f\in R_q[\underline{X}]$ tel que $f^n \in IR_q[\underline{X}]$ avec $n>0$. Comme $I$ est radical, on vérifie immédiatement qu'il en est de même de $IK[\underline{X}]$. Cela montre que $IK_q[\underline{X}]$ est aussi radical, car l'extension $K_q/K$ est séparable. On déduit alors que $f \in IK_q[\underline{X}]$. Cela signifie qu'il existe $r\geq 1$ tel que $t^{r/q}f \in IR_q[\underline{X}]$. Mais $IR_q[\underline{X}]$ est $t^{1/q}$-saturé, et donc $f\in IR_q[\underline{X}]$. Cela prouve que $IR_q[\underline{X}]$ est bien radical.
\end{itemize}
\end{proof}

\hspace{3ex} Pour montrer un résultat similaire à celui du théorème \ref{Green} pour $K_{\infty}$, nous avons besoin de travailler simultanément avec tous les $K_q$, qui sont tous redevables du théorème \ref{Green}. Plus précisément, si l'on fixe un système d'équations polynomiales sur $R$, nous pouvons le voir comme un système d'équations à coefficients dans $R_q$ pour chaque $q\geq 1$: le théorème \ref{Green} fournit alors des entiers $N_q$, $ c_q$ et $ s_q$, et notre objectif dans la suite est de contrôler ces entiers quand on fait varier $q$. Il est donc assez naturel d'introduire la définition technique suivante:

\begin{definition}\label{def}
Soit $\underline{F}=(F_1,...,F_r)$ un système de $r$ polynômes à $n$ variables à coefficients dans $R$. 
\begin{itemize}
\item[(i)] Soit $q \in \mathbb{N}^*$. On dit qu'un triplet $(N,c,s) \in \mathbb{N}^* \times \mathbb{N}^* \times \mathbb{N} $ est associé à $(R,t,q,\underline{F})$ s'il vérifie la propriété suivante: pour chaque $\nu \geq N$ et chaque $\underline{x} \in R_q^n$ tel que $$\underline{F}(\underline{x}) \equiv 0 \mod t^{\nu/q},$$ il existe $\underline{y} \in R_q^n$ tel que $$\underline{y} \equiv \underline{x} \mod t^{([\nu /c]-s)/q} \;\;\ \text{et} \;\;\; \underline{F}(\underline{y})=0.$$
\item[(ii)] On dit qu'un  quadruplet $(q_0,N,c,s)\in\mathbb{N}^*\times \mathbb{N}^*\times\mathbb{N}^*\times \mathbb{N} $ est $(R,t,\underline{F})$-admissible si, pour tout $q\geq 1$, le triplet $(qN,c,qs)$ est associé à $(R,t,qq_0,\underline{F})$. 
\end{itemize}
\end{definition}

\hspace{3ex} Nous sommes maintenant en mesure d'énoncer le théorème suivant:

\begin{theorem}\label{thinfty}
Soit $R$ un anneau de valuation discrète hensélien de corps des fractions $K$. On suppose $K$ de caractéristique 0 et on se donne une uniformisante $t$ de $R$. Pour $q\geq 1$, on pose $K_q=K(t^{1/q})$ et $R_q=\mathcal{O}_{K_q}$. On pose aussi $K_{\infty} = \bigcup_{q\geq 1} K_q$ et $R_{\infty}= \bigcup_{q\geq 1} R_q$. Soit $\underline{F}=(F_1,...,F_r)$ un système de $r$ polynômes à $n$ variables à coefficients dans $R$. Alors il existe un quadruplet $(q_0,N,c,s)$ qui est $(R,t,\underline{F})$-admissible.
\end{theorem}

Pour établir ce théorème, nous allons fortement utiliser les constructions utilisées dans la preuve du théorème \ref{Green} (voir \cite{Gre}). 

\begin{proof}
On note $V$ la $K$-variété affine définie par $\underline{F}=0$ et on appelle $m$ sa dimension. On va montrer par récurrence sur $m$ qu'il existe un quadruplet d'entiers $(R,t,\underline{F})$-admissible. 
\begin{itemize}
\item[$\bullet$] Si $m=-1$ (ie $V=\emptyset$), alors il existe $u \in (R\cap \underline{F}R[\underline{X}]) \setminus \{0\}$. Le quadruplet $(1,\text{val}_R (u)+1,1,0)$ est alors $(R,t,\underline{F})$-admissible, puisque pour ces valeurs de $q_0,N,c,s$, l'hypothèse apparaissant dans la définition \ref{def}(i) n'est pas vérifiée.
\item[$\bullet$] Supposons $m \geq 0$.
\begin{itemize}
\item[$\circ$] Supposons d'abord que l'idéal $\underline{F}R[\underline{X}]$ est radical et $t$-saturé, et que $V_{K_{\infty}}$ est irréductible. Dans ce cas, le lemme \ref{sature} montre alors que l'idéal $\underline{F}R_q[\underline{X}]$ de $R_q[\underline{X}]$ est radical pour tout $q\geq 1$. Soient $J$ la matrice jacobienne de $\underline{F}$ et $\underline{D}$ le système des mineurs d'ordre $n-m$ de $J$. Par hypothèse de récurrence, il existe un quadruplet $(q_0',N',c',s')$ qui est $(R,t,\underline{F},\underline{D})$-admissible. Pour $I \subseteq \{1,...,r\}$ avec $|I|=n-m$, notons $\underline{F}_I$ le système constitué des polynômes $F_i$ pour $i \in I$. Soit $V_{I}$ la $K$-variété définie par le système $\underline{F}_I=0$. Soit $V_I^{+}$ la réunion des composantes irréductibles de $V_I$ qui sont de dimension $m$ et différentes de $V$. Soit $\underline{G}_{I}$ un système de générateurs de l'idéal de $V_I^+$ dans $R[\underline{X}]$. Par hypothèse de récurrence, il existe un quadruplet $(q_{0,I},N_I,c_I,s_I)$ qui est $(R,t,\underline{G}_I,\underline{F})$-admissible. Posons:
$$q_0=q_0'\prod\limits_{\underset{I \subseteq \{1,...,n\}}{|I|=n-m}} q_{0,I}.$$
D'après la preuve du théorème 1 de \cite{Gre}, pour chaque $q \geq 1$, le triplet $(N^{(q)},c^{(q)},s^{(q)})$ avec:
\begin{gather*}
N^{(q)}=2+2qq_0 \max \left\lbrace \frac{N'}{q_0'},\max \left\lbrace\frac{N_I}{q_{0,I}}\vert I \subseteq \{1,...,n\}, |I|=n-m\right\rbrace\right\rbrace \\
c^{(q)}=2\max \left\lbrace c', \max \left\lbrace c_I \vert I \subseteq \{1,...,n\}, |I|=n-m\right\rbrace\right\rbrace\\
s^{(q)}=1+qq_0 \max \left\lbrace \frac{s'}{q_0'},\max \left\lbrace\frac{s_I}{q_{0,I}}\vert I \subseteq \{1,...,n\}, |I|=n-m\right\rbrace\right\rbrace
\end{gather*}
est associé à $(R,t,qq_0,\underline{F})$. On en déduit que le quadruplet $(q_0,N,c,s)$ avec:
\begin{gather*}
N=2+2q_0 \max \left\lbrace \frac{N'}{q_0'},\max \left\lbrace\frac{N_I}{q_{0,I}}\vert I \subseteq \{1,...,n\}, |I|=n-m\right\rbrace\right\rbrace \\
c=2\max \left\lbrace c', \max \left\lbrace c_I \vert I \subseteq \{1,...,n\}, |I|=n-m\right\rbrace\right\rbrace\\
s=1+q_0 \max \left\lbrace \frac{s'}{q_0'},\max \left\lbrace\frac{s_I}{q_{0,I}}\vert I \subseteq \{1,...,n\}, |I|=n-m\right\rbrace\right\rbrace
\end{gather*}
est $(R,t,\underline{F})$-admissible.
\item[$\circ$] On ne fait plus d'hypothèses. Soit $q_0'\geq 1$ tel que les composantes irréductibles $W_1,...,W_u$ de $V_{K_{q_0'}}$ restent irréductibles sur $K_{\infty}$. Pour chaque $j\in \{1,...,u\}$, soit $I'_j$ l'idéal premier de $K_{q_0'}[\underline{X}]$ définissant $W_j$. Considérons:
$$I_j:= I'_j \cap R_{q_0'}[\underline{X}].$$
Soit $\underline{G}_j$ un système de générateurs de $I_j$. On vérifie aisément que $I_j$ est radical et $t^{1/q_0'}$-saturé. De plus, la $K_{q_0'}$-variété définie par $I_j$ est $W_j$: c'est une variété de dimension au plus $m$ et $(W_j)_{K_{\infty}}$ est irréductible. On en déduit qu'il existe un quadruplet $(q_{0,j},N_j,c_j,s_j)$ qui est $(R_{q_0'},t^{1/q_0'},\underline{G}_j)$-admissible. 
Remarquons maintenant qu'il existe $w \in \mathbb{N}^{*}$ tel que $(I'_1...I'_u)^w \subseteq \underline{F}K_{q_0'}[\underline{X}]$. On en déduit qu'il existe $v \in \mathbb{N}^{*}$ tel que:
\begin{equation} \label{incl}
t^{uvw/q_0'} (I_1...I_u)^w  \subseteq \underline{F}R_{q_0'}[\underline{X}].
\end{equation}
 Posons:
$$q_0=q_0' \prod_j q_{0,j},$$
et donnons-nous un entier $q\geq 1$. Notons $\text{val}: R_{qq_0} \rightarrow \mathbb{Z}$ la valuation sur $R_{qq_0}$, et considérons les entiers:
\begin{gather*}
N^{(q)}=uw \frac{qq_0}{q_0'}\left( \max \left\lbrace \frac{N_j}{q_{0,j}}\right\rbrace + v\right)  \\
c^{(q)} = uw\max \{c_j\}\\
s^{(q)} = 1+\frac{qq_0}{q_0'}\left( v+\max \left\lbrace\frac{s_{j}}{q_{0,j}}\right\rbrace\right) .
\end{gather*}
Soient $\nu \geq N^{(q)}$ et $\underline{x} \in R_{qq_0}^n$ tels que $\underline{F}(\underline{x}) \equiv 0 \mod t^{\nu /(qq_0)}$. Si l'on se donne des polynômes $g_1\in I_1R_{qq_0}[\underline{X}],...,g_u\in I_uR_{qq_0}[\underline{X}]$, l'inclusion (\ref{incl}) implique que:
$$\nu \leq \text{val} \left[ t^{uvw/q_0'}\left( \prod_{j=1}^u g_j(\underline{x})  \right)^w  \right] = \frac{qq_0}{q_0'}uvw + w \sum_j \text{val}(g_j(\underline{x})).   $$
Il existe donc $j_0 \in \{1,...,u\}$ tel que:
$$\text{val}(g_{j_0}(\underline{x})) \geq \frac{\nu}{uw} - \frac{qq_0}{q_0'}v.$$
Cela étant vrai quels que soient les polynômes $g_1\in I_1R_{qq_0}[\underline{X}],...,g_u\in I_uR_{qq_0}[\underline{X}]$ choisis, on déduit qu'il existe $j_1 \in \{1,...,u\}$ tel que:
\begin{equation} \label{ineg1}
\forall g \in I_{j_1}R_{qq_0}[\underline{X}], \text{val}(g(\underline{x})) \geq \frac{\nu}{uw} - \frac{qq_0}{q_0'}v.
\end{equation}
Comme $\nu \geq N^{(q)}$, on a aussi:
\begin{equation} \label{ineg2}
\frac{\nu}{uw} - \frac{qq_0}{q_0'}v \geq \frac{qq_0}{q_0'q_{0,j_0}}N_{j_0}.
\end{equation}
Comme le quadruplet $(q_{0,{j_0}},N_{j_0},c_{j_0},s_{j_0})$ est $(R_{q_0'},t^{1/q_0'},\underline{G}_{j_0})$-admissible, le triplet $(\frac{qq_0}{q_0'q_{0,j}}N_{j_0},c_{j_0},\frac{qq_0}{q_0'q_{0,{j_0}}}s_{j_0})$ est associé à $\left( R_{q_0'},t^{1/q_0'},\frac{qq_0}{q_0'},\underline{G}_{j_0} \right) $. On déduit alors de (\ref{ineg1}) et de (\ref{ineg2}) qu'il existe $\underline{y}\in R_{qq_0}^n$ tel que:
$$\underline{y} \equiv \underline{x} \mod t^{\mu/(qq_0)} \;\;\; \text{et} \;\;\; \underline{G}_{j_0}(\underline{y}) = 0,$$
où $\mu=[\frac{\nu}{c_{j_0}uw}]-\frac{qq_0}{q_0'c_{j_0}}v-1-\frac{qq_0}{q_0'q_{0,j_0}}s_{j_0}$. On en déduit que:
$$\underline{y} \equiv \underline{x} \mod t^{([\nu/c^{(q)}]-s^{(q)})/(qq_0)} \;\;\; \text{et} \;\;\; \underline{F}(\underline{y}) = 0,$$
et donc que le triplet $(N^{(q)},c^{(q)},s^{(q)})$ est associé à $(R,t,qq_0,\underline{F})$. Cela prouve que le quadruplet $(q_0,N,c,s)$ avec:
\begin{gather*}
N=uw \frac{q_0}{q_0'}\left( \max \left\lbrace \frac{N_j}{q_{0,j}}\right\rbrace + v\right)  \\
c = uw\max \{c_j\}\\
s = 1+\frac{q_0}{q_0'}\left( v+\max \left\lbrace\frac{s_{j}}{q_{0,j}}\right\rbrace\right) 
\end{gather*}
est $(R,t,\underline{F})$-admissible.
\end{itemize}
\end{itemize}
\end{proof}

\begin{corollary}\label{thcro}
Soit $R$ un anneau de valuation discrète hensélien de corps des fractions $K$. On suppose $K$ de caractéristique 0 et on se donne une uniformisante $t$ de $R$. Pour $q\geq 1$, on pose $K_q=K(t^{1/q})$ et $R_q=\mathcal{O}_{K_q}$. On pose aussi $K_{\infty} = \bigcup_{q\geq 1} K_q$ et $R_{\infty}= \bigcup_{q\geq 1} R_q$. Soit $\underline{F}=(F_1,...,F_r)$ un système de $r$ polynômes à $n$ variables à coefficients dans $R_{\infty}$.
Il existe $M \in \mathbb{Q}^+$, $\gamma \in \mathbb{N}^*$ et $\sigma \in \mathbb{Q}^+$ vérifiant la propriété suivante: pour chaque rationnel $\mu \geq M$ et chaque $\underline{x} \in R_{\infty}^n$ tel que $$\underline{F}(\underline{x}) \equiv 0 \mod t^{\mu},$$ il existe $\underline{y} \in R_{\infty}^n$ tel que $$\underline{y} \equiv \underline{x} \mod t^{\mu /\gamma-\sigma} \;\;\ \text{et} \;\;\; \underline{F}(\underline{y})=0.$$
\end{corollary}

\begin{proof}
Quitte à remplacer $R$ par $R_q$ pour un certain $q$ assez grand, on peut supposer que le système $\underline{F}$ est à coefficients dans $R$. D'après le théorème \ref{thinfty}, il existe un quadruplet $(R,t,\underline{F})$-admissible $(q_0,N,c,s)$. Posons $M=N/q_0$, $\gamma = c$ et $\sigma =\frac{s+1}{q_0}$. Considérons $\mu \in\mathbb{Q}$ tel que $\mu \geq M$ et écrivons $\mu =a/b$ avec $a,b\in \mathbb{N}^*$. Supposons qu'il existe $\underline{x} \in R_{\infty}^n$ tel que $\underline{F}(\underline{x}) \equiv 0 \mod t^{\mu}$. Soit $q_1  \geq 1$ tel que $\underline{x}\in R_{q_1}^n$. On sait que, pour tout $q\geq 1$, le triplet $(qN,c,qs)$ est associé à $(R,t,qq_0,\underline{F})$. En particulier, le triplet $(bq_1N,c,bq_1s)$ est associé à $(R,t,bq_1q_0,\underline{F})$. Comme $\underline{F}(\underline{x}) \equiv 0 \mod t^{\mu}$ et $\mu \geq N/q_0$, cela montre qu'il existe $\underline{y}\in R_{bq_0q_1}^n$ tel que $\underline{F}(\underline{y})=0$ et:
$$\underline{y} \equiv \underline{x} \mod t^{\lambda}$$
avec $\lambda = \frac{1}{bq_1q_0}\left( \left[ \frac{aq_1q_0}{c}\right] -bq_1s \right) $. Cela achève la preuve puisque $\lambda \geq \frac{\mu}{c}-\sigma$.
\end{proof}

\begin{corollary}\label{infty}
Sous les hypothèses du théorème \ref{thcro}, si la congruence $\underline{F}(\underline{x}) \equiv 0 \mod t^{\nu}$ a des solutions dans $R_{\infty}$ pour chaque entier $\nu\geq 1$, alors l'équation $\underline{F}(\underline{x})=0$ a des solutions dans $R_{\infty}$.
\end{corollary}

\subsection{Le corps $\mathbb{C}(x)((t))$ est $C_1^1$}\label{lolo}

\hspace{3ex} Nous sommes finalement en mesure d'établir les conjectures de Kato et Kuzumaki pour les corps de valuation discrète complets dont le corps résiduel est le corps des fonctions d'une variété sur un corps algébriquement clos de caractéristique 0:

\begin{theorem}\label{c(x)((t))}
Soient $k$ un corps algébriquement clos de caractéristique nulle et $m\geq 1$. Alors le corps $k(t_1,...,t_m)((t))$ est $C_i^j$ pour tous $i \geq 0$ et $j \geq 0$ tels que $i+j=m+1$.
\end{theorem}

\begin{proof}
Le corps $k(t_1,...,t_m)((t))$ étant $C_{m+1}$ et de dimension cohomologique $m+1$, on peut supposer $j\neq 0$ et $i \neq 0$. Soient $Y$ une $k$-variété intègre de dimension $m$ et $K=k(Y)((t))$. Dans la suite, on fixe une clôture algébrique $\overline{K}$ de $K$. Tous les corps seront vus dans $\overline{K}$. Soit $Z$ une hypersurface de $\mathbb{P}^n_K$ de degré $d$, avec $d^i\leq n$.
\begin{itemize}
\item[$\bullet$] Soit $(u_1,...,u_j) \in {k(Y)^{\times}}^j$. Soit $(v_1,...,v_r)$ une base de transcendance de l'extension $k(Y)/k(u_1,...,u_j)$. Soit $M$ la clôture algébrique de $k(u_1,...,u_j,v_{m-j+1},...,v_{r})$ dans $K$: c'est un corps de degré de transcendance $j$ sur $k$. Soit $Y'$ une $M$-variété géométriquement intègre de dimension $i-1$ telle que $k(Y)=M(Y')$. Le corps $\overline{M}(Y')$ est un corps $C_{i-1}$, et donc le corps:
$$K_M:= \bigcup_{F/M \text{ finie}} F(Y')((t))$$
est $C_{i}$. On en déduit que $Z(K_M) \neq \emptyset$, et donc qu'il existe une extension finie $F$ de $M$ telle que $Z(F(Y')((t)))\neq \emptyset$. Comme $M$ est $C_0^j$, on a $\{u_1,...,u_j\} \in N_{F/M}(K_j^M(F))$, et donc $\{u_1,...,u_j\} \in N_j(Z/K)$.
\item[$\bullet$]  Soit $(u_1,...,u_{j-1}) \in {k(Y)^{\times}}^{j-1}$. Soit $f\in k(Y)[[t]][X_0,...,X_n]$ un polynôme homogène définissant $Z$.  Soit $(v_1,...,v_r)$ une base de transcendance de l'extension $k(Y)/k(u_1,...,u_{j-1})$. Soit $M$ la clôture algébrique de $k(u_1,...,u_{j-1},v_{m-j+2},...,v_{r})$ dans $K$: c'est un corps de degré de transcendance $j-1$ sur $k$. Soit $Y'$ une $M$-variété géométriquement intègre de dimension $i$ telle que $k(Y)=M(Y')$. On introduit:
\begin{gather*}
K_M:= \bigcup_{F/M \text{ finie}} F(Y')((t)),\\
R_M:= \bigcup_{F/M \text{ finie}} F(Y')[[t]].
\end{gather*}
L'anneau $R_M$ est un anneau de valuation discrète hensélien de corps des fractions $K_M$, d'uniformisante $t$ et de corps résiduel $\overline{M}(Y')$. On pose aussi:
\begin{gather*}
K_{\infty}:=\bigcup_{q\geq 1} K_M(t^{1/q}),\\
R_{\infty}: = \bigcup_{q\geq 1} R_M[t^{1/q}].
\end{gather*}
Soient $\mathfrak{m}_{\infty}$ l'idéal maximal de $R_{\infty}$ et $\nu \geq 1$ un entier. Soient $f_{\nu} \in M((t))(Y')[X_0,...,X_n]$ et $g_{\nu} \in R_{\infty}[X_0,...,X_n]$ homogènes de degré $d$ tels que:
$$f=f_{\nu}+t^{\nu}g_{\nu}.$$
Comme $\overline{M((t))}(Y')$ est $C_i$ et est contenu dans $K_{\infty}$, il existe $(x_0,...,x_n)\in  R_{\infty}^{n+1} \setminus \mathfrak{m}_{\infty}^{n+1}$ tel que $f_{\nu}(x_0,...,x_n)=0$. On a du coup:
$$f(x_0,...,x_n) \equiv 0 \mod t^{\nu}.$$
Cela étant vrai pour tout $\nu \geq 1$, on déduit du corollaire \ref{infty} que $Z(K_{\infty}) \neq \emptyset$. Considérons alors une extension finie $F/M$ et un entier $q\geq 1$ tels que $Z(F(Y')((t^{1/q})))\neq \emptyset$. Comme $M$ est un corps $C_0^{j-1}$, il existe $x\in K_{j-1}^M(F)$ tel que $N_{F/K}(x) = \{u_1,...,u_{j-1}\}$. Par conséquent:
\begin{align*}
N_{F(Y')((t^{1/q}))/K}(\{x,t^{1/q}\}) & = N_{F(Y')((t))/M(Y')((t))}\left(  N_{F(Y')((t^{1/q}))/F(Y')((t))}(\{x,t^{1/q}\})\right) \\& =\pm N_{F(Y')((t))/M(Y')((t))} ( \{x,t\})\\&=\pm \{u_1,...,u_{j-1},t\}.
\end{align*}
On en déduit que $\{u_1,...,u_{j-1},t\}\in N_j(Z/K)$.
\end{itemize}
Comme le groupe $K_j^M(K)/d$ est engendré par les symboles de la forme $\{u_1,...,u_j\}$ et $\{u_1,...,u_{j-1},t\}$ avec $(u_1,...,u_j) \in {k(Y)^{\times}}^j$, on déduit que $N_j(Z/K) = K^{\times}$.
 \end{proof}
 
 \begin{remarque}~
 \begin{itemize}
\item[$\bullet$] Soient $k$ un corps algébriquement clos de caractéristique nulle et $Y$ une $k$-variété intègre de dimension $m$. La preuve précédente montre en fait que, si $M$ est une extension de degré de transcendance $j-1$ sur $k$ contenue dans $k(Y)$ et si $Y'$ est une $M$-variété intègre telle que $M(Y')=k(Y)$, alors le corps:
$$K_{\infty} = \bigcup_{q\geq 1} \bigcup_{\text{$F/M$ finie}} F(Y')((t))(t^{1/q})$$
est $C_{m+1-j}$. En particulier, le corps $\bigcup_{q\geq 1}\mathbb{C}(x)((t))(t^{1/q})$ est $C_1$.
\item[$\bullet$] Le théorème précédent ne peut pas être établi en utilisant les techniques de \cite{Wit}, puisque $\mathbb{C}(x)((t))$ n'est pas $C^1_1$-fort.
\item[$\bullet$] Soit $L$ le corps des fonctions d'une courbe projective lisse géométriquement intègre $C$ sur $\mathbb{C}((t))$. Soit $\mathcal{C}$ un modèle propre régulier de $C$ sur $\mathbb{C}[[t]]$ dont la fibre spéciale est un diviseur à croisements normaux stricts. Si l'on veut démontrer que $L$ est $C_1^1$ par des arguments de type principe local-global, le théorème précédent permet d'utiliser toutes les places de $L$ provenant d'un point de codimension 1 de $\mathcal{C}$ (et pas seulement celles qui proviennent d'un point de codimension 1 de $C$). La même remarque s'applique si l'on prend pour $L$ une extension finie de $\mathbb{C}((x,y))$, pour $C$ le spectre de la clôture intégrale de $\mathbb{C}[[x,y]]$ dans $L$, et pour $\mathcal{C}$ une résolution des singularités de $C$ dont la fibre spéciale est un diviseur à croisements normaux stricts.
\end{itemize}
 \end{remarque}

\section{Vers certains corps de fonctions}

\hspace{3ex} Dans cette partie, nous allons nous intéresser aux corps $\mathbb{C}((t))(x)$ et $\mathbb{C}((x,y))$, ou plus généralement à $\mathbb{C}((t_1))...((t_m))(x)$ et $\mathbb{C}((t_1))...((t_{m-1}))((x,y))$. Pour ce faire, nous avons besoin de commencer par établir le lemme suivant, qui sera ensuite appliqué à certains complétés de $\mathbb{C}((t_1))...((t_m))(x)$ et de $\mathbb{C}((t_1))...((t_{m-1}))((x,y))$:

\begin{lemma}\label{ramif}
Soient $k$ un corps algébriquement clos de caractéristique nulle et $m$ un entier naturel. Soient $K=k((t_1))...((t_m))$ et $f\in K[X]$ un polynôme irréductible unitaire de degré premier $d$ tel que $f^{(d-1)}(0)=0$. Alors $f(0) \not\in {K^{\times}}^d$.
\end{lemma}

\begin{proof}
On procède par récurrence sur $m$. Le résultat est évident si $m=0$. On suppose donc le résultat démontré pour un certain entier naturel $m$ et on considère l'anneau $R=k((t_1))...((t_m))[[t_{m+1}]]$ et le corps $K=k((t_1))...((t_{m+1}))$. Soit $\text{val}: R \rightarrow \mathbb{Z}$ la valuation sur l'anneau de valuation discrète $R$. On se donne $f\in K[X]$ un polynôme irréductible unitaire de degré premier $d$ tel que $f^{(d-1)}(0)=0$. Quitte à faire un changement de variable, on peut supposer que $f$ est à coefficients dans $R$. Soit $L$ un corps de rupture de $f$. L'extension $L/K$ est galoisienne de degré $d$. La théorie de Kummer montre alors qu'il existe $u \in R $ tel que $L=K(\sqrt[d]{u})$ et $\text{val}(u)\in\{0,...,d-1\}$. Soit $x\in L^{\times}$ tel que $f(x)=0$, et écrivons $x=x_0+x_1\sqrt[d]{u}+...+x_{d-1}\sqrt[d]{u}^{d-1}$ avec $x_0,...,x_{d-1}\in K$. Comme $f^{(d-1)}(0)=0$, on a $x_0=0$. Deux cas se présentent:
\begin{itemize}
\item[$\bullet$] \textit{Premier cas:} $\text{val}(u)\neq 0$. Dans ce cas, si l'on étend la valuation $\text{val}$ à $L$, on remarque que, pour $i \in \{1,...,d-1\}$, on a:
$$\text{val}(x_{i}\sqrt[d]{u}^{i}) \in \frac{i}{d}\text{val}(u) + \mathbb{Z}.$$
Comme $d$ ne divise pas $\text{val}(u)$, on déduit que les rationnels $\text{val}(x_{i}\sqrt[d]{u}^{i})$ (pour $i \in \{1,...,d-1\}$) sont deux à deux distincts. Soit $j \in \{1,...,d-1\}$ tel que $\text{val}(x_{j}\sqrt[d]{u}^{j})$ soit minimal. On a alors:
$$\text{val}(f(0))=\text{val}(N_{L/K}(x))=d\text{val}(x_j)+j\text{val}(u) \not\in d\mathbb{Z}.$$
On en déduit que $f(0) \not\in {K^{\times}}^d$.
\item[$\bullet$] \textit{Second cas:} $\text{val}(u)= 0$. Dans ce cas, la réduction de $f$ modulo $t_{m+1}$ est irréductible. Par hypothèse de récurrence, la réduction de $f(0)$ modulo $t_{m+1}$ n'est pas dans ${k((t_1))...((t_m))^{\times}}^d$. On en déduit que $f(0) \not\in {K^{\times}}^d$.
\end{itemize} 
\end{proof}

\hspace{3ex} Avant de passer au théorème principal de cette section, faisons quelques rappels sur les complexes motiviques. Dans l'article \cite{Blo}, Bloch associe à chaque corps $E$ et à chaque entier naturel $i$ un complexe de modules galoisiens noté $z^i(E,\cdot)$. Soit $\mathbb{Z}(i)$ le complexe $z^i(-,\cdot)[-2i]$. Dans la suite, nous aurons besoin notamment des résultats suivants:

\begin{theorem}\label{motiv} Soient $E$ un corps et $i$ un entier naturel.
\begin{itemize}
\item[(i)] On a $H^{i+1}(E,\mathbb{Z}(i))=0$.
\item[(ii)] On a un isomorphisme canonique $K_i^M(E) \cong H^i(E,\mathbb{Z}(i))$.
\end{itemize}
\end{theorem}

Ces résultats découlent du théorème de Nesterenko-Suslin (théorème 0.3(v) de \cite{Izq1}) et de la conjecture de Beilinson-Lichtenbaum (théorème 0.3(vii) de \cite{Izq1}). Ils impliquent notamment la proposition suivante:

\begin{proposition}\label{normmot}
Soit $E$ un corps de caractéristique nulle. Soient $E_1,...,E_n$ des extensions finies de $E$. Soit $F=E_1\times ... \times E_n$. Soit $\check{T}$ le module des cocaractères du tore $T=R^1_{F/E}(\mathbb{G}_m)$. Pour $i \geq 0$, on a un isomorphisme canonique:
$$H^{i+1}(E,\check{T} \otimes \mathbb{Z}(i)) \cong \frac{K_i^M(E)}{\prod_{j=1}^n N_{E_j/E}(K_i^M(E_j))}.$$
\end{proposition}

\begin{proof}
Le triangle distingué:
$$\check{T} \otimes \mathbb{Z}(i) \rightarrow \left( \bigoplus_{j=1}^n \mathbb{Z}[\text{Gal}(\overline{E}/E)/\text{Gal}(\overline{E}/E_j)]\right) \otimes \mathbb{Z}(i) \rightarrow \mathbb{Z}(i) \rightarrow \check{T} \otimes \mathbb{Z}(i)[1] $$
induit une suite exacte:
$$\bigoplus_{j=1}^n H^i(E_j,\mathbb{Z}(i)) \rightarrow H^i(E,\mathbb{Z}(i)) \rightarrow H^{i+1}(E,\check{T} \otimes \mathbb{Z}(i)) \rightarrow \bigoplus_{j=1}^n H^{i+1}(E_j,\mathbb{Z}(i)), $$
où la première flèche est donnée par la corestriction. En tenant compte du théorème \ref{motiv}, cette suite exacte s'identifie à la suite exacte:
$$\bigoplus_{j=1}^n K_i^M(E_j) \rightarrow K_i^M(E) \rightarrow H^{i+1}(E,\check{T} \otimes \mathbb{Z}(i)) \rightarrow 0, $$
où la première flèche est donnée par la norme.
\end{proof}

\hspace{3ex} Nous sommes maintenant en mesure d'énoncer et d'établir le théorème suivant:

\begin{theorem}\label{diff}
Soient $k$ un corps algébriquement clos de caractéristique nulle et $m, i, j$ des entiers naturels tels que $i+j=m+1$. Soient $L$ le corps des fonctions d'une courbe projective lisse géométriquement intègre sur $k((t_1))...((t_m))$ ou encore une extension finie de $k((t_1))...((t_{m-1}))((x,y))$. Soit $Z$ une hypersurface de degré $d$ dans $\mathbb{P}^n_L$ avec $d^i \leq n$ et $d$ premier. Alors:
\begin{itemize}
\item[(i)] Il existe des extensions finies $L_1$, ..., $L_r$ de $L$ telles que $Z(L_s)\neq \emptyset$ pour chaque $s$ et $N_j(Z/L) = \prod_{s=1}^r N_{L_s/L}(K_j^M(L_s))$.
\item[(ii)] Le groupe $K_j^M(L)/N_j(Z/L)$ est fini.
\end{itemize}
\end{theorem}

\begin{notation} ~
\begin{itemize}
\item[$\bullet$] Si $L$ est le corps des fonctions d'une courbe projective lisse géométriquement intègre sur $k((t_1))...((t_m))$, on notera $C$ une courbe projective lisse géométriquement intègre sur $k((t_1))...((t_m))$ telle que $L=k((t_1))...((t_m))(C)$. Si $L$ est une extension finie de $k((t_1))...((t_{m-1}))((x,y))$, on notera $C$ le spectre de la clôture intégrale de $k((t_1))...((t_{m-1}))[[x,y]]$ dans $L$.
\item[$\bullet$] On notera $C^{(1)}$ l'ensemble des points de codimension 1 de $C$. Pour $v\in C^{(1)}$, la notation $L_v$ désignera le complété de $L$ en $v$. Finalement, lorsque $M$ est un module galoisien (ou un complexe de modules galoisiens) sur $L$, on notera:
$$\Sha^1(L,M) := \text{Ker} \left( H^1(L,M) \rightarrow \prod_{v\in C^{(1)}} H^1(L_v,M) \right) .$$
\end{itemize}
\end{notation}

\begin{proof}
Soit $f \in L[X_0,...,X_n]$ un polynôme homogène de degré $d$ tel que $Z$ est l'hypersurface définie par l'équation $f=0$. Pour $0 \leq \alpha < \beta \leq n$, considérons le polynôme:
$$f_{\alpha  \beta}(X) = f(0,...,0,X,0,...,0,1,0,...,0)$$
où $X$ est à la $\alpha$-ème place et $1$ à la $\beta$-ème place.\\

Si $Z$ a un 0-cycle de degré 1, il n'y a rien à démontrer. On suppose donc que $Z$ n'a pas de 0-cycles de degré 1. Dans ce cas, les $f_{\alpha  \beta}$ sont tous irréductibles de degré $d$, et un changement de variable linéaire permet de supposer aussi que $f_{\alpha  \beta}^{(d-1)}(0)=0$ pour chaque $\alpha<\beta$.\\

Notons maintenant $L_{\alpha  \beta}$ le corps de rupture de $f_{\alpha  \beta}$ et fixons $v\in C^{(1)}$. Décomposons chaque $f_{\alpha  \beta}$ en produit d'irréductibles sur $L_v$:
$$f_{\alpha  \beta} = \prod_{\gamma =1}^{r_{\alpha  \beta}} f_{\alpha  \beta}^{(\gamma)}.$$
Notons $M_{\alpha  \beta}^{(\gamma)}$ le corps de rupture de $f_{\alpha  \beta}^{(\gamma)}$ sur $L_v$ et montrons que:
\begin{equation}\label{wre}
K_j^M(L_v) = \prod_{\alpha, \beta, \gamma} N_{M_{\alpha  \beta}^{(\gamma)}/L_v}\left(K_j^M(M_{\alpha  \beta}^{(\gamma)})\right).
\end{equation}

Si l'un des $f_{\alpha  \beta}$ n'est pas irréductible sur $L_v$, les degrés $[M_{\alpha  \beta}^{(\gamma)}:L_v]$ sont premiers entre eux dans leur ensemble, et donc l'égalité (\ref{wre}) est évidente. On peut donc supposer que chaque $f_{\alpha  \beta}$ est irréductible sur $L_v$. On notera alors $M_{\alpha  \beta}$ au lieu de $M_{\alpha  \beta}^{(1)}$. \\

Donnons-nous $u_1,...,u_j$ des éléments de $L_v$ et montrons que le symbole $\{u_1,...,u_j\} \in K_j^M(L_v)$ appartient à $\prod_{\alpha, \beta} N_{M_{\alpha  \beta}/L_v}\left(K_j^M(M_{\alpha  \beta})\right)$. Deux cas se présentent:
\begin{itemize}
\item[$\bullet$] \textit{Premier cas:} les classes de $u_1,...,u_j$ dans le $\mathbb{Z}/d\mathbb{Z}$-espace vectoriel $L_v^{\times}/{L_v^{\times}}^d$ sont linéairement liées. Sans perte de généralité, on peut supposer que:
$$u_j = u_1^{\delta_1}...u_{j-1}^{\delta_{j-1}} u^d$$
pour certains $\delta_1,...\delta_{j-1}\in \mathbb{Z}$ et $u \in L_v$. On a alors:
\begin{align*}
\{u_1,...,u_j\} &= \{u_1,...,u_{j-1},u\}^d\prod_{s=1}^{j-1} \{u_1,...,u_{j-1},u_s\}^{\delta_s}\\
&= \{u_1,...,u_{j-1},u\}^d \prod_{s=1}^{j-1} \left( \{u_1,...,u_{j-1},-u_s\}\{u_1,...,u_{j-1},-1\} \right) ^{\delta_s}\\
&= \left(\{u_1,...,u_{j-1},u\}\prod_{s=1}^{j-1} \{u_1,...,u_{j-1},\zeta_{2d}\}  ^{\delta_s}\right)^d,
\end{align*}
où $\zeta_{2d}$ désigne une racine primitive $2d$-ième de l'unité. On en déduit que le symbole $\{u_1,...,u_j\} \in K_j^M(L_v)$ est $d$-divisible et donc qu'il appartient bien à $\prod_{\alpha, \beta} N_{M_{\alpha  \beta}/L_v}\left(K_j^M(M_{\alpha  \beta})\right)$. 
\item[$\bullet$] \textit{Second cas:} les classes de $u_1,...,u_j$ dans le $\mathbb{Z}/d\mathbb{Z}$-espace vectoriel $L_v^{\times}/{L_v^{\times}}^d$ sont linéairement indépendantes. Dans ce cas, le sous-$\mathbb{Z}/d\mathbb{Z}$-espace vectoriel $W$ de $L_v^{\times}/{L_v^{\times}}^d$ engendré par les classes de $u_1,...,u_j$ est de dimension $j$. Pour $0 \leq \alpha \leq n$, notons $c_{\alpha}$ le coefficient de $X_{\alpha}^p$ dans $f$. Pour chaque $\alpha$ et $\beta$ tels que $\alpha<\beta$, il existe $x_{\alpha  \beta} \in M_{\alpha  \beta}^{\times}$ tel que $c_{\alpha}c_{\beta}^{-1} = N_{M_{\alpha  \beta}/L_v}(x_{\alpha  \beta})$. De plus, d'après le lemme \ref{ramif}, les classes des $c_{\alpha}$ dans $L_v^{\times}/{L_v^{\times}}^d$ sont deux à deux distinctes. Comme $n \geq d^i$, cela montre que le sous-$\mathbb{Z}/d\mathbb{Z}$-espace vectoriel $W'$ de $L_v^{\times}/{L_v^{\times}}^d$ engendré par les $c_{\alpha}c_{\beta}^{-1}$ (avec $\alpha < \beta$) est de dimension au moins $i+1=m+2-j$. On a donc: $$\dim_{\mathbb{Z}/d\mathbb{Z}} W + \dim_{\mathbb{Z}/d\mathbb{Z}} W' \geq j+(m+2-j) = m+2 > m+1= \dim_{\mathbb{Z}/d\mathbb{Z}} L_v^{\times}/{L_v^{\times}}^d,$$
d'où $\dim_{\mathbb{Z}/d\mathbb{Z}} (W \cap W') > 0$. Soit $w$ un élément de $L_v^{\times}$ dont la classe modulo ${L_v^{\times}}^d$ est un élément non trivial de $W \cap W'$. Écrivons: $$w=u_1^{\delta_1}...u_j^{\delta_j}u^d=\left( \prod_{\alpha < \beta}(c_{\alpha}c_{\beta}^{-1})^{\epsilon_{\alpha \beta}}\right) u'^d$$ avec $(\delta_1,...,\delta_j)\in \mathbb{Z}^j \setminus (d \mathbb{Z})^j$, $\epsilon_{\alpha \beta} \in \mathbb{Z}$ pour chaque $\alpha < \beta$ et $u,u' \in L_v^{\times}$. Sans perte de généralité, on peut supposer que $\delta_j \not\in d\mathbb{Z}$. Comme $c_{\alpha}c_{\beta}^{-1} = N_{M_{\alpha  \beta}/L_v}(x_{\alpha  \beta})$ pour chaque $\alpha < \beta$, on a:
\begin{align*}
\{u_1,...,u_j\}^{\delta_j} & = \{u_1,...,u_{j-1},w\} \{u_1,...,u_{j-1},w^{-1}u_j^{\delta_j}\}\\
& =\left(  \prod_{\alpha < \beta} \{u_1,...,u_{j-1},c_{\alpha}c_{\beta}^{-1}\}^{\epsilon_{\alpha \beta}}\right)  \{u_1,...,u_{j-1},w^{-1}u_j^{\delta_j}u'^d\}\\
&=\left(  \prod_{\alpha < \beta} N_{M_{\alpha  \beta}/L_v}(\{u_1,...,u_{j-1},x_{\alpha \beta}\})^{\epsilon_{\alpha \beta}}\right)  \{u_1,...,u_{j-1},w^{-1}u_j^{\delta_j}u'^d\}.
\end{align*} 
De plus, d'après le premier cas:
$$\{u_1,...,u_{j-1},w^{-1}u_j^{\delta_j}u'^d\} \in \prod_{\alpha < \beta} N_{M_{\alpha  \beta}/L_v}\left(K_j^M(M_{\alpha  \beta})\right).$$
Comme $d$ ne divise pas $\delta_j$, cela prouve que:
$$\{u_1,...,u_j\} \in \prod_{\alpha < \beta} N_{M_{\alpha  \beta}/L_v}\left(K_j^M(M_{\alpha  \beta})\right).$$
\end{itemize}
Nous avons donc montré l'égalité (\ref{wre}).\\

Or, d'après la proposition \ref{normmot}, cette égalité implique que le groupe $K_j^M(L)/N_j(Z/L)$ est un quotient de $$\Sha^{j+1}(L,\check{T} \otimes \mathbb{Z}(j)),$$ où $\check{T}$ désigne le module des cocaractères du tore $R^{1}_{\prod_{\alpha, \beta} L_{\alpha  \beta}/L}(\mathbb{G}_m)$. Le théorème découle alors de la finitude du groupe $\Sha^{j+1}(L,\check{T} \otimes \mathbb{Z}(j))$ (théorème 3.10 de \cite{Izq1} et variante de la remarque 2.9 de \cite{Izq2}).
\end{proof}

\begin{remarque}\label{rqdiff}
Le théorème s'applique aux extensions finies de $\mathbb{C}((x))(y)$ et aux extensions finies de $\mathbb{C}((x,y))$ en prenant $i=j=1$.
\end{remarque}

\section{Questions}

\hspace{3ex} Dans ce paragraphe, nous soulevons un certain nombre de questions qui se posent naturellement à partir des résultats des sections précédentes. Nous ne prétendons pas que leurs réponses soient affirmatives.

\subsection{Corps locaux}

\begin{question}\label{q1}
Soient $K$ un corps $p$-adique et $t$ une uniformisante de $K$. Soit $L= \bigcup_{q \geq 0} K(t^{1/q})$. Peut-on utiliser le corollaire \ref{infty} pour démontrer que $L$ est $C_{1}$?
\end{question}

\begin{remarque}~
\begin{itemize}
\item[$\bullet$] Le corps $L$ est bien de dimension cohomologique 1.
\item[$\bullet$] Dans le cas $K=\mathbb{Q}_p$ et $t=p$, il s'agit de démontrer que, si $f$ est un polynôme homogène à $n+1$ variables de degré $d$ avec $d \leq n$ à coefficients dans $\mathbb{Z}$, alors pour tout $\nu \geq 0$, la congruence $f(x_0,...,x_n)\equiv 0 \mod p^{\nu}$ a des solutions dans $\bigcup_{q\geq 1}\mathbb{Z}[p^{1/q}]$ telles que les $x_i$ ne sont pas tous divisibles par une puissance de $p$.
\item[$\bullet$] Une réponse affirmative à la question \ref{q1} permettrait de retrouver la propriété $C_1^1$ pour les corps $p$-adiques. Pour ce faire, il suffirait d'ailleurs de montrer que $L$ est $C_1^0$.
\end{itemize}
\end{remarque}

\begin{question}\label{q2}
Soient $n\geq 1$ un entier naturel et $J$ une partie de $\{1,...,n\}$. Considérons le corps:
$$K_J=\bigcup_{(a_j)_{j \in J} \in \mathbb{N}^J} \mathbb{C}((x_1))...((x_n))(x_j^{1/a_j}, j \in J).$$ Peut-on utiliser le corollaire \ref{infty} pour démontrer que $K_J$ est $C_{n-|J|}$?
\end{question}

\begin{remarque}~
\begin{itemize}
\item[$\bullet$] Le corps $K_J$ est de dimension cohomologique $n-|J|$, puisque son groupe de Galois absolu est isomorphe à $\widehat{\mathbb{Z}}^{n-|J|}$.
\item[$\bullet$] En s'inspirant de la preuve du théorème \ref{c(x)((t))} et en vertu du corollaire \ref{infty}, pour répondre à la question \ref{q2}, il suffirait de démontrer que $K_J$ contient un sous-corps $C_{n-|J|}$ contenant $\mathbb{C}((x_1))...((x_{n-1}))(x_j^{1/a_j}, j \in J\setminus \{n\})(x_n)$. Par exemple, pour montrer que $K_{\{n\}}$ est $C_{n-1}$, il suffirait de montrer qu'il contient un sous-corps $C_{n-1}$ contenant $\mathbb{C}((x_1))...((x_{n-1}))(x_n)$.
\item[$\bullet$] Une réponse affirmative à la question \ref{q2} impliquerait que le corps $\mathbb{C}((x_1))...((x_n))$ est $C_i^j$ pour tous $i \geq 0$ et $j\geq 0$ tels que $i+j=n$. Pour ce faire, il suffirait d'ailleurs de démontrer la propriété $C_{n-|J|}^0$ pour $K_J$.
\end{itemize}
\end{remarque}

\begin{question}\label{q3}
Avec les notations du théorème \ref{thinfty}, si le corps résiduel de $K$ est $C_i$ peut-on déduire que $K_{\infty}$ est $C_i$? 
\end{question}

\begin{remarque}~
\begin{itemize}
\item[$\bullet$] D'après le corollaire \ref{infty}, si le corps résiduel $k$ de $K$ est de caractéristique 0, il s'agit de démontrer que, si $f$ est un polynôme homogène à $n+1$ variables de degré $d$ avec $d^i \leq n$ à coefficients dans $k[[t]]$, alors pour tout $\nu \geq 1$, la congruence $f(x_1,...,x_n) \equiv 0 \mod t^{\nu}$ a une solution dans $\bigcup_{q\geq 1} k[[t^{1/q}]]$ telle qu'aucune puissance de $t$ divise tous les $x_i$.
\item[$\bullet$] Toujours en tenant compte du corollaire \ref{infty}, pour répondre à la question \ref{q3}, il suffit de démontrer la propriété $C_i$ pour les hypersurfaces \textit{lisses} sur $K_{\infty}$, auquel cas il est possible d'utiliser le lemme 7.5 de \cite{Wit}. Cette remarque s'applique d'ailleurs aussi aux questions \ref{q1} et \ref{q2}.
\item[$\bullet$] Une réponse affirmative à la question \ref{q3} impliquerait des réponses affirmatives aux questions \ref{q1} et \ref{q2}.
\end{itemize}
\end{remarque}

\begin{question}
Avec les notations du corollaire \ref{infty}, si $K$ est $C_i^0$ peut-on déduire que $K_{\infty}$ est $C_i^0$?
\end{question}

\subsection{Corps globaux}

\begin{question}
L'énoncé du théorème \ref{diff} reste-t-il vrai lorsque l'hypersurface considérée n'est plus de degré premier?
\end{question}

\begin{question}\label{que}
Soit $L$ le corps des fonctions d'une courbe projective lisse géométriquement intègre $C$ sur $\mathbb{C}((t))$. Soit $\mathcal{C}$ un modèle propre régulier de $C$ sur $\mathbb{C}[[t]]$ dont la fibre spéciale est un diviseur à croisements normaux stricts. Soient $L_1,...,L_n$ des extensions finies de $L$. Soit $L'$ une extension finie de $L$ linéairement disjointe avec l'extension composée $L_1...L_n$. Soient $E= L' \times \prod_{i=1}^n L_i$ et $T=R^{1}_{E/K}(\mathbb{G}_m)$. Le groupe:
$$\Sha^1_{\mathcal{C}}(L,T) := \text{Ker} \left( H^1(L,T) \rightarrow \prod_{v \in \mathcal{C}^{(1)}} H^1(L_v,T) \right) $$
est-il forcément nul?
\end{question}

\begin{question}
Même question que \ref{que} en prenant pour $L$ une extension finie de $\mathbb{C}((x,y))$, pour $C$ le spectre de la clôture intégrale de $\mathbb{C}[[x,y]]$ dans $L$, et pour $\mathcal{C}$ une résolution des singularités de $C$ dont la fibre spéciale est un diviseur à croisements normaux stricts.
\end{question}

\begin{remarque}~
\begin{itemize}
\item[$\bullet$] Les deux questions précédentes demandent si un résultat analogue au théorème 1 de \cite{DW} reste valable sur $\mathbb{C}((t))(x)$ ou sur $\mathbb{C}((x,y))$. Pour pouvoir utiliser les mêmes méthodes que \cite{DW}, on aurait besoin de montrer l'annulation de:
$$\Sha^1(L,T) := \text{Ker} \left( H^1(L,T) \rightarrow \prod_{v \in C^{(1)}} H^1(K_v,T) \right) $$
puisqu'il est connu que le groupe $\Sha^1(L,T)$ vérifie un théorème de dualité à la Poitou-Tate (alors que ce n'est pas connu pour $\Sha^1_{\mathcal{C}}(L,T)$). Mais la preuve du théorème 1 de \cite{DW} utilise fortement le théorème de Chebotarev, qui tombe en défaut sur $\mathbb{C}((t))(x)$ et sur $\mathbb{C}((x,y))$: en effet, le groupe
$$\Sha^2(L,\mathbb{Z}) := \text{Ker} \left( H^2(L,\mathbb{Z}) \rightarrow \prod_{v \in C^{(1)}} H^2(K_v,\mathbb{Z}) \right) $$
peut être non nul. Il n'est donc pas possible d'adapter immédiatement la preuve de \cite{DW}.
\item[$\bullet$] Une réponse affirmative aux deux questions précédentes montrerait que $\mathbb{C}((t))(x)$ et $\mathbb{C}((x,y))$ vérifient la propriété $C^1_1$: il suffirait de procéder exactement comme dans la section \ref{s1}, à condition d'utiliser que les corps $\mathbb{C}((x))((y))$ et $\mathbb{C}(x)((y))$ sont $C_1^1$ (voir la section 4 de \cite{Wit} et la section \ref{lolo} du présent article).
\end{itemize}
\end{remarque}

\begin{question}
Soient $K=\mathbb{C}((x,y))$ et $L$ le corps obtenu à partir de $K$ en adjoignant toutes les racines de $x$. Le corps $L$ est $C_1$?
\end{question}

\begin{remarque}~
\begin{itemize}
\item[$\bullet$] Il serait déjà intéressant de savoir si $L$ est de dimension cohomologique 1.
\item[$\bullet$] Si la réponse à la question précédente était affirmative, il serait possible de montrer que $\mathbb{C}((x,y))$ est $C_1^1$ par un argument similaire à celui du théorème \ref{tresglobal}.
\end{itemize}
\end{remarque}


\begin{thebibliography}{CTGP04}

\footnotesize

\bibitem[Blo86]{Blo}
Spencer Bloch.
\newblock Algebraic cycles and higher {$K$}-theory.
\newblock {\em Adv. in Math.}, 61(3):267--304, 1986.

\bibitem[CTM04]{CTM}
Jean-Louis Colliot-Th\'el\`ene and David A. Madore.
\newblock Surfaces de del {P}ezzo sans point rationnel sur un corps de dimension cohomologique un.
\newblock {\em J. Inst. Math. Jussieu}, 3(1):1--16, 2004.

\bibitem[DW14]{DW}
Cyril Demarche and Dasheng Wei.
\newblock Hasse principle and weak approximation for multinorm equations.
\newblock {\em Israel J. Math.}, 202(1):275--293, 2014.

\bibitem[ELW15]{ELW}
Hélène Esnault, Marc Levine and Olivier Wittenberg.
\newblock Index of varieties over {H}enselian fields and {E}uler characteristic of coherent sheaves.
\newblock {\em J. Algebraic Geom.}, 24(4):693--718, 2015.

\bibitem[FJ08]{FA}
Michael D. Fried and Moshe Jarden.
\newblock {\em Field arithmetic}, 3ème éd.
\newblock Ergeb. Math. Grenzgeb. (3), 11, Springer, Berlin, 2008.

\bibitem[Gre66]{Gre}
Marvin J. Greenberg.
\newblock Rational points in {H}enselian discrete valuation rings.
\newblock {\em Inst. Hautes \'Etudes Sci. Publ. Math.}, 31:59--64, 1966.

\bibitem[Izq16a]{Izq1}
Diego Izquierdo.
\newblock Théorèmes de dualité pour les corps de fonctions sur des corps locaux supérieurs.
\newblock {\em Mathematische Zeitschrift}, 284(1-2), p. 615-642, 2016.

\bibitem[Izq16b]{Izq2}
Diego Izquierdo.
\newblock Dualité et principe local-global sur des corps locaux de dimension 2.
\newblock Prépublication sur \url{http://www.math.ens.fr/~izquierd/recherche.html}.

\bibitem[Kat80]{Kat}
Kazuya Kato.
\newblock A generalization of local class field theory by using $K$-groups. II.
\newblock {\em J. Fac. Sci. Univ. Tokyo Sect. IA Math.}, 27(3):603--683, 1980.

\bibitem[KK86]{KK}
Kazuya Kato and Takako Kuzumaki.
\newblock The dimension of fields and algebraic {$K$}-theory.
\newblock {\em J. Number Theory}, 24(2):229--244, 1986.

\bibitem[KS83]{KS}
Kazuya Kato and Shuji Saito.
\newblock Unramified class field theory of arithmetical surfaces.
\newblock {\em Ann. of Math. (2)}, 118(2):241--275, 1983.

\bibitem[Kol07]{Kol}
J\'anos Koll\'ar.
\newblock A conjecture of {A}x and degenerations of {F}ano varieties.
\newblock {\em Israel J. Math.}, 162:235--251, 2007.

\bibitem[Mer91]{Mer}
Alexander S. Merkurjev.
\newblock Simple algebras and quadratic forms.
\newblock {\em Izv. Akad. Nauk SSSR Ser. Mat.}, 55(1):218--224, 1991.

\bibitem[Wit15]{Wit}
Olivier Wittenberg.
\newblock Sur une conjecture de {K}ato et {K}uzumaki concernant les hypersurfaces de {F}ano.
\newblock {\em Duke Math. J.}, 164(11):2185--2211, 2015.

\end{thebibliography}
\end{document}